\documentclass[11pt]{article}
\usepackage{amsmath,amsthm,amsfonts,amssymb,amscd, amsxtra, mathrsfs,textcomp,verbatim,inputenc}
\usepackage{url}
\usepackage{xcolor}
\usepackage[margin=2.5 cm,nohead]{geometry}
\theoremstyle{plain}
\theoremstyle{definition}
\newtheorem{theorem}{Theorem}
\newtheorem{lemma}{Lemma}
\newtheorem{definition}{Definition}
\newtheorem{corollary}{Corollary}
\newtheorem{proposition}{Proposition}

\newtheorem{remark}{Remark}
\newtheorem{example}{Example}
\DeclareMathOperator{\Hess}{Hess}
\DeclareMathOperator{\grad}{grad}

\newcommand{\N}{\mathbb N}
\newcommand{\lng}{\left\langle}
\newcommand{\rng}{\right\rangle}
\newcommand{\lf}{\left}
\newcommand{\rg}{\right}

\begin{document}
%%%%%%%%%%%%%%%%%%%%%%%%%%%%%%%%%%%%%%%%%%%%%%%%%%%%%%%%%%%%%%%%%%%%%%%%%%%%%%%%%%%%%%%%%%%%%%%%%%%%%%%%%%%%%%
\title{Convexity of sets and quadratic functions on the hyperbolic space}
\author{
 O. P. Ferreira\thanks{IME/UFG, Avenida Esperan\c{c}a, s/n, Campus II,  Goi\^ania, GO - 74690-900, Brazil (E-mail: {\tt orizon@ufg.br}).}
\and
S. Z. N\'emeth \thanks{School of Mathematics, University of Birmingham, Watson Building, Edgbaston, Birmingham - B15 2TT, United Kingdom
(E-mails: {\tt s.nemeth@bham.ac.uk}, {\tt JXZ755@bham.ac.uk}).}
\and
J. Zhu %\thanks{School of Mathematics, University of Birmingham, Watson Building, Edgbaston, Birmingham - B15 2TT, United Kingdom
%(E-mails: {\tt s.nemeth@bham.ac.uk}, {\tt JXZ755@bham.ac.uk}).} 
\footnotemark[3]
}

\maketitle

\begin{abstract}
In this paper some concepts of convex analysis on hyperbolic spaces are studied.
We first   study   properties of  the  intrinsic distance, for instance, we
present the spectral  decomposition of its Hessian. Next,  we   study the
concept of convex sets and the intrinsic projection onto these sets.  We also
study the concept of convex functions and present first and second order
characterizations of these functions, as well as some optimization concepts
related to them. An  extensive study  of the  hyperbolically convex quadratic
functions is also presented.
\vspace{5mm}

\vspace{2mm}

\end{abstract}
\section{Introduction}
%%%%%%%%%%%%%%%%%%%%
The modeling of various classes of constrained optimization problems via Riemannian manifolds has gained increasing attention in the recent years  in both the academic and business community as,  in some sense,   it outperforms the Euclidean counterpart.  Besides  the purely theoretical motivations and interests, it aims to obtain practical tools for supporting efficient computational implementations of algorithms for solving such problems. A point that deserves to be highlighted is that endowing  the set of constraints of the  problem  with  a suitable Riemannian metric, it  allows us to explore its intrinsics  algebraic and  geometric structures  and then  significantly reduce the cost of obtaining solutions of  the problem in question. For example, a non-convex Euclidean problem can be seen as a convex Riemannian one (as we show in Section~\ref{sec:hcqf}), whose optimization methods for solving it have much less inherent computational complexity, see for instance \cite[Example 13.4.1]{Rapcsak1997} and \cite{Ferreira2019, Ferreira2020}.    It is worth noting that  the concepts of convexity of  sets and functions in the  Riemannian optimization context is a topic that is of interest in itself, see for example  \cite{Kristaly2016, Rapcsak1991, Udriste1994, Zhou2017}.

The hyperbolic space is a non-Euclidean  smooth manifold of negative constant
sectional curvature, see for example  \cite{BenedettiPetronio1992,
Ratcliffe2019}. A concise introductory note on  hyperbolic spaces can be found in
\cite{Cannon1997}. Over the recent years several  theoretical and practical
applications  of the hyperbolic space  have emerged. Although we are not
concerned with practical issues at this time, we emphasize that practical
applications appear whenever the natural structure of the data is modeled as an
optimization problem on the hyperbolic space. For instance,  several problems in
machine learning, artificial intelligence,   financial networks, as well as
procrustes  problems and many other practical questions   can be modeled  in this
setting.  Papers dealing  with these subjects  include, for machine learning
\cite{nickel2018learning}, for artificial intelligence
\cite{muscoloni2017machine}, for neural circuits  \cite{sharpee2019argument},
for low-rank  approximations  of  hyperbolic  embeddings
\cite{jawanpuria2019low,tabaghi2020hyperbolic}, for    procrustes  problems
\cite{Tabaghi2021}, for financial networks \cite{keller2021hyperbolic}, for
complex networks \cite{Kriouko2010, Moshir2021}, for embeddings of data \cite{Wilson2014} and the references therein.   We  also mention that  there are many related papers on  strain analysis, see for example, \cite{vollmer2018automatic,yamaji2008theories}.  
  
The aim of this paper is to study  some concepts related to the  convexity of
sets and functions  on the hyperbolic space in an intrinsic way. Although some of
these concepts have already been studied in general Riemannian manifolds, we
revisit them in this specific context in order to present some explicit  formulas
and new properties.  To this end, among the various existing models of hyperbolic
geometry, we choose the hyperboloid model (also called Lorentz model), see
\cite{BenedettiPetronio1992, Cannon1997}. We first   study  important properties
of  the  intrinsic distance, for instance, we present the spectral  decomposition
of its Hessian. Next,  we   study   the concept of convex sets and the intrinsic
projection onto these sets.  We also study the concept of convex functions and
present the first and second order characterizations  for  these functions.
Finally, we  present an  extensive study  of  the hyperbolically convex quadratic functions.

The structure of this paper is as follows. In Section~\ref{sec:int.01}, we recall
some notations and basic results.  In Section~\ref{sec:int.1}, we recall some
notations, definitions and basic properties  about the geometry of the hyperbolic
space used throughout  the paper. In Section~\ref{sec:pid}  we present some
properties of the intrinsic  distance from a fixed point.  In
Section~\ref{sec:cs} we present a characterization  of convex sets in the
hyperbolic space and in Section~\ref{sec:pj} we study properties of the
projection onto convex sets.  In Section~\ref{sec:cf} we study the basic
properties of convex functions on the hyperbolic space and in
Section~\ref{sec:hcqf} we study hyperbolically convex quadratic functions. In
Section~\ref{sec:fr2} we present some concepts of optimization related to
hyperbolically convex functions.   We  conclude this paper by  making some final remarks in 
Section~\ref{sec:fr}.
%%%%%%%%%%%%%%%%%%%%%%%%%%%%%%%%%%%%%%%%%%%%%%%%%%%%%%%%%%%%%%%%%
\subsection{Notation and Basics Results} \label{sec:int.01}
 Let ${\mathbb R}^{m}$ be   the $m$-dimensional Euclidean space. The set of all $m \times n$ matrices with real entries is denoted by ${\mathbb R}^{m \times n}$ and ${\mathbb R}^m\equiv {\mathbb R}^{m\times 1}$. For 
 $M \in {\mathbb R}^{m\times n}$ the matrix $M^{\top}  \in {\mathbb R}^{n\times m}$  denotes  the  {\it transpose} of $M$ and $\lambda_{\min}(M)$ and $\lambda_{\max}(M)$ stands for the minimum and maximum  eigenvalue of the matrix  $M$, respectively.   If $x\in {\mathbb R}^m$, then ${\rm diag} (x)\in {\mathbb R}^{m \times m}$ denotes a  diagonal matrix with $(i,i)$-th entry equal to $x_i$, $i=1,\dots,n$.   The matrix ${\rm I}$ denotes the $n\times n$ identity matrix.  The following characterizations of symmetric positive definite matrices can be found in \cite[Theorem 7.2.5, pp. 404]{HornJohnson1990}.
 \begin{lemma} \label{eq:posdef}
 Let  $M\in {\mathbb R}^{n\times n}$ be symmetric.  Then,  $M$ is  positive
 definite if and only if $\det M_i>0$, for $i=1,\ldots, n$, where $M_i$ stands
 for the principal submatrix  $M$ determined by the first $i$ rows and first $i$ columns. 
  \end{lemma}
For a given  $M\in{\mathbb R}^{(n+1)\times(n+1)}$,  consider the following decomposition:
\begin{equation}\label{eq:notf}
	M:=\left(
	\begin{matrix}
	{\bar M} & {b}\\
	{b}^\top & \theta
	\end{matrix}
	\right),  \qquad   {\bar M}\in {\mathbb R}^{n\times n}, \quad {b}\in {\mathbb R}^{n\times 1}, \quad \theta \in {\mathbb R}.
\end{equation} 
  Using the decomposition  \eqref{eq:notf} we have the following
  characterizations  of positive definite  matrices and positive semidefinite matices,  for the proof see \cite[Propositions~16.1 and 16.2]{Gallier2011}.
    \begin{lemma} \label{le:aspd}
 Let  $M=M^\top \in{\mathbb R}^{(n+1)\times(n+1)}$. Consider the decomposition of $M$ in the form \eqref{eq:notf}.
  \begin{itemize}
\item[(i)] $M$ is positive definite if and only if $\theta>0$ and ${\bar M}-\frac{1}{\theta}bb^\top$ is positive definite;
\item[(ii)] If $\theta>0$, then $M$ is positive semidefinite if and only if ${\bar M}-\frac{1}{\theta}bb^\top$ is positive semidefinite;
\item[(iii)] If  ${\bar M}$ is positive  definite, then $M$ is positive semidefinite if and only if $\theta-b^\top{\bar M}^{-1}b\geq 0$.
 \end{itemize}
   \end{lemma}
Now, assume that $ {\bar M}$ is invertible. In this case the determinant of $M$ is given  
by the formula of the next lemma, see  \cite[Section 0.85, pp. 21]{HornJohnson1990}.
 \begin{lemma} \label{le:sscdet}
 Let  $M\in{\mathbb R}^{(n+1)\times(n+1)}$. Consider the decomposition of $M$ in the form \eqref{eq:notf}.  Then,  
 $\det(M)=(\theta-b^\top {\bar M}^{-1}b)\det{\bar M}$.
   \end{lemma}
   Next we present a version of the S-lemma which can be found for example  in  \cite[Theorem 2.2]{PolikTerlaky2007}.
 \begin{lemma} \label{eq:slemma}
 Let  $A, B \in  {\mathbb R}^{n\times n}$ be symmetric matrices. Assume that
 ${\hat x}^{\top}B{\hat x}<0$  for some ${\hat x} \in {\mathbb R}^{n\times n}$.
 Then  the following two statements are equivalent.
 \begin{itemize}
\item[(i)] If ${x}^{\top}B{x}\leq 0$, then  $ {x}^{\top}A{x}\geq 0$; 
\item[(ii)] There exists a $\beta\geq 0$ such that $A+\beta B$ is  positive semidefinite.
 \end{itemize}
 \end{lemma}
   We end this section by stating  a version of  Finsler's lemma, see  \cite{Finsler1936}. A 
   proof of it can be found, for example, in   \cite[Theorem 2]{Marcellini1984}.
   \begin{lemma} \label{le:Finslert}
 Let  $M, N\in{\mathbb R}^{n\times n}$ be  two symmetric matrices. If $x^\top N x=0$ implies  $x^\top M x>0$, then there exists $\lambda \in {\mathbb R}$ such that  $M+\lambda N$  is positive definite.
   \end{lemma}
%%%%%%%%%%%%%%%%%%%%%%%%%%%%%%%%%%%%%%%%%%%%%%%%%%%%%%%%%%%%%%%%%
\section{Basics Results About the Hyperbolic Space} \label{sec:int.1}
%%%%%%%%%%%%%%%%%%%%%%%%%%%%%%%%%%%%%%%%%%%%%%%%%%%%%%%%%%%%%%%%%
In this section we recall some notations, definitions and basic properties about the geometry  of the hyperbolic space used throughout the paper. They can be found in many introductory books on  Riemannian and Differential  Geometry, for example in \cite{BenedettiPetronio1992, Ratcliffe2019}, see also \cite{Boumal2020}. 

Let $\langle\cdot , \cdot \rangle_{\tiny}$ be the {\it Lorentzian inner product}  of  $x:=(x_1, \ldots,x_n, x_{n+1})^{\top} $ and  $y:=(y_1, \ldots, y_n,y_{n+1})^{\top}$ on ${{\mathbb R}^{n+1}}$ defined  as follows
\begin{equation} \label{eq:ip}
{\langle} x, y{\rangle}:= x_{1}y_{1}+ \cdots + x_{n}y_{n}-x_{n+1}y_{n+1}.
\end{equation}
For each  $x\in {{\mathbb R}^{n+1}}$,  the {\it Lorentzian norm (length)} of $x$ is defined to be the complex number
\begin{equation} \label{eq:norm}
\|x\|:= \sqrt{{\langle} x, x{\rangle}}.
\end{equation}
Here $\|x\|$ is either positive, zero, or positive imaginary. In order to state
the inner product \eqref{eq:ip} in a convenient form, we  take the diagonal matrix ${\rm J}$ 
defined  by
\begin{equation} \label{eq:dmip}
{\rm J}:={\rm diag}(1, \ldots,1, -1) \in {\mathbb R}^{(n+1)\times (n+1)}.
\end{equation}
By using \eqref{eq:dmip},   the  Lorentz inner product  \eqref{eq:ip}  can be stated equivalently  as follows 
\begin{equation} \label{eq:ipef}
{\langle} x, y{\rangle}:=x^{\top}{\rm J} y, \qquad \forall x, y\in  {{\mathbb R}^{n+1}}.
\end{equation}
 Throughout the paper the {\it $n$-dimensional  hyperbolic space} and its {\it tangent hyperplane at a point $p$} are denoted  by
\begin{equation} \label{eq:hs}
{\mathbb H}^{n}:=\left\{ p\in {{\mathbb R}^{n+1}}:~\langle p, p\rangle
=-1, ~p^{n+1}>0\right\}, \qquad T_{p}{{\mathbb H}^n}:=\left\{v\in {{\mathbb R}^{n+1}}\, :\, 
\langle p, v \rangle=0\right\},
\end{equation}
respectively.
It is worth noting that  the  Lorentzian inner product defined in  \eqref{eq:ip}
is not  positive definite in the entire space ${{\mathbb R}^{n+1}}$.  However,
one can show that its restriction to the tangent spaces of ${\mathbb H}^{n}$ is positive 
definite; see \cite[Section 7.6]{Boumal2020}. Consequently,  $\|v\|>0$ for all
$v\in T_{p}{{\mathbb H}^n}$ and  all $p\in {\mathbb H}^{n}$ with $v\neq 0$.
Therefore, {\it  $\langle \cdot, \cdot\rangle$  and  $\|\cdot\|$ are  in fact  a
positive inner product and the associated  norm in $ T_{p}{{\mathbb H}^n}$, for
all $p\in {\mathbb H}^{n}$}. Next we present a basic lemma used in the sequel.
\begin{lemma} \label{le:fl}
Let $p, q\in {\mathbb H}^{n}$. Then,  $\langle p , q\rangle\leq -1$ and  $\langle p , q\rangle= -1$ if and only if $p=q$.
\end{lemma}
\begin{proof}
Since  $p, q\in {\mathbb H}^{n}$, we have  $\langle p,  p\rangle = -1$, $
p_{n+1}>0$,  $\langle q, q\rangle = -1$ and $q_{n+1}>0$. Thus, we have
$p_{n+1}=\sqrt{u^{\top}u}$ and $q_{n+1}=\sqrt{v^{\top}v}$, where  $u=\left(p_1,
\ldots, p^n, 1\right)^{\top}$ and $v=(q_1, \ldots, q_n, 1)^{\top}$.
Hence, it  follows from \eqref{eq:ip} that 
$
     \langle p, q\rangle=p_1q_1+\dots +p_{n}q_{n}-\sqrt{u^{\top}u}\sqrt{v^{\top}v}.
$
On the other hand,  by taking into account that Cauchy's inequality in the
Euclidean space implies that $\sqrt{u^{\top}u} \sqrt{v^{\top}v}\geq u^{\top}v$
and the equality holds if and  only if $u=v$, the result follows.  
\end{proof}
Therefore,  \eqref{eq:ip} actually defines a Riemannian metric on ${\mathbb H}^{n}$, see \cite[pp. 67]{Cannon1997}.  The   {\it  Lorentzian projection onto the  tangent hyperplane} $T_p{{\mathbb H}^n}$ is the linear mapping defined by
\begin{equation} \label{eq: proj}
 {\rm I}+pp^{\top}{\rm J} : {{\mathbb R}^{n+1}} \to T_p{{\mathbb H}^n}, \qquad \forall p\in  {{\mathbb H}^n},
\end{equation}
where ${\rm I}  \in {\mathbb R}^{(n+1)\times (n+1)}$ is the identity matrix.
\begin{remark} \label{eq:sadj}
The  Lorentzian projection \eqref{eq: proj} is self-adjoint with respect to the   Lorentzian inner product \eqref{eq:ip}. Indeed, $\langle ( {\rm I}+pp^{\top}{\rm J} )u, v\rangle=\langle u, ({\rm I}+pp^{\top}{\rm J})v\rangle $,  for all   $u, v \in  {{\mathbb R}^{n+1}}$ and all $p\in  {{\mathbb H}^n}$. Moreover, we also have $({\rm I}+pp^{\top}{\rm J})({\rm I}+pp^{\top}{\rm J})={\rm I}+pp^{\top}{\rm J}$, for all $p\in  {{\mathbb H}^n}$.
\end{remark} 
The  {\it intrinsic distance on the  hyperbolic space} between two  points $p, q \in {{\mathbb H}^n}$  is  defined by
\begin{equation} \label{eq:Intdist}
d(p, q):={\rm arcosh} (-\langle p , q\rangle).
\end{equation}
 It can  be shown that   $({{\mathbb H}^n}, d)$  is a complete metric space, so that $d(p,q)\ge 0$ for all $p,q \in {{\mathbb H}^n}$, and 
$d(p,q)=0$ if and only if $p=q$.   Moreover,  $({{\mathbb H}^n}, d)$ has the same topology as ${{\mathbb R}^{n}}$. The intersection curve of a plane though the origin of ${{\mathbb R}^{n+1}}$ 
with  $ {{\mathbb H}^n}$ is called a { \it geodesic}.   Moreover,  each  geodesic segment  
$\gamma: [a, b]\to {{\mathbb H}^n}$ is  {\it minimal},  i.e., its arc length is  equal to the
intrinsic  distance $\ell(\gamma)={\rm arcosh} (-\langle  \gamma(a), \gamma(b)\rangle)$ between its end points. 
 We say that $ \gamma $ is a {\it normalized geodesic}  if $\| \gamma ^{\prime}\|=1$. 
If $ p, q\in {{\mathbb H}^n}$ and  $q\neq p$, then the   unique {\it geodesic segment  from $p$ to $q$ } is   
\begin{equation*}
\gamma_{pq}(t)= \left( \cosh t + \frac{\langle p, q\rangle \sinh t}{\sqrt{\langle
p, q\rangle^2-1}}\right) p 
+ \frac{\sinh t}{\sqrt{\langle p, q\rangle^2-1}}\;q, \qquad \forall t\in [0, \;d(p,q)].
\end{equation*}
The {\it exponential mapping} $\exp_{p}:T_{p}{{\mathbb H}^n} \rightarrow {{\mathbb H}^n} $ is defined 
by $\exp_{p}v\,=\, \gamma _{v}(1)$, where $\gamma _{v}$ is the geodesic defined by its 
initial position $p$, with velocity $v$ at $p$. Hence, $\exp_{p}v=p$ for  $v=0$,  and 
\begin{equation*} 
\exp_{p}v:= \displaystyle \cosh(\|v\|) \,p+ \sinh(\|v\|)\, \frac{v}{\|v\|},
\qquad  \forall v\in T_p{{\mathbb H}^n}\setminus\{0\}.
\end{equation*} 
It is easy to prove that $\gamma _{tv}(1)=\gamma _{v}(t)$ for any values of $t$. 
Therefore,  for all $t\in {\mathbb R}$ we have
\begin{equation} \label{eq:geoexp}
\exp_{p}tv:=  \displaystyle \cosh(t\|v\|) \,p+ \sinh(t\|v\|)\, \frac{v}{\|v\|},  \qquad  \forall v\in T_p{{\mathbb H}^n}/\{0\}.
\end{equation}
We will also use the expression above for denoting  the geodesic starting 
at $p\in {{\mathbb H}^n}$ with velocity $v\in T_p{{\mathbb H}^n}$ at $p$.  
The {\it inverse of the exponential mapping}  is given by $\log_{p}q=0$, for $q=p$, and 
\begin{equation} \label{eq:expinv}
\log_{p}q:=  \displaystyle \frac{{\rm arcosh}(-\langle p, q\rangle)}{\sqrt{\langle p, q\rangle ^2-1}} 
\left[{\rm I}+pp^\top{\rm J}\right]q,  \qquad  q\neq p.
\end{equation}
It follows from \eqref{eq:Intdist} and \eqref{eq:expinv} that
\begin{equation*} 
d(p, q)=\|\log_{p}q\|, \qquad  p, q \in {{\mathbb H}^n}.
\end{equation*}
Let $\Omega \subseteq {{\mathbb H}^n}$ be an  open set and $f: \Omega \to  {\mathbb R}$ a differentiable function.  The {\it gradient on the hyperbolic space} of $f$  is the unique vector field $\Omega \ni p\mapsto \grad f(p)\in T_{p}{\cal M}$ such that $df(p)v=\langle \grad f(p), v\rangle$, see \cite[Proposition~7-5, p.162]{Boumal2020}.  Therefore, we have
\begin{equation} \label{eq:grad}
\grad f(p):= \left[{\rm I}+pp^{\top}{\rm J} \right] {\rm J}\cdot Df(p)= {\rm J}\cdot Df(p)+ \langle {\rm J}\cdot Df(p), p\rangle \,p,
\end{equation}
where $Df(p) \in {\mathbb R}^{n+1}$ is the usual gradient of  $ f$ at $p\in \Omega$. 
A {\it vector field} on $\Omega \subseteq {{\mathbb H}^n}$ is a smooth 
mapping $ X: \Omega \to {\mathbb R}^{n+1}$ such that $X(p)\in  T_{p}{{\mathbb H}^n}$.  
The  { \it covariant  derivative} of $X $  at  $p\in \Omega$ is map   $ \nabla X(p):T_p{{\mathbb H}^n} \to T_p{{\mathbb H}^n}$ given by
\begin{equation*} 
\nabla X(p):=\left[{\rm I}+pp^{\top}{\rm J} \right]DX (p),
\end{equation*}
where  $DX (p)$ denotes   the usual derivative  of  the vector field $X$ at the point $p$, see \cite[Formula~(7.62),  p.162]{Boumal2020}.  The {\it Hessian on the hyperbolic space} 
of a twice differentiable function $f: \Omega \to  {\mathbb R}$ at a point $p\in \Omega$ is  the  mapping 
$ \nabla \grad f (p):= \Hess f(p):T_p{{\mathbb H}^n} \to T_p{{\mathbb H}^n}$ given by
\begin{equation} \label{eq:Hess}
\Hess f(p):= \left[{\rm I}+pp^{\top}{\rm J} \right]\left[{\rm J}\cdot D^2f(p)+\langle {\rm J}\cdot Df(p), p\rangle {\rm I}\right],
\end{equation}
where $ D^2f(p)$ is the usual Hessian (Euclidean Hessian) of the function  $f$ at a point $p$, see  \cite[Proposition~7.6, p.163]{Boumal2020}.  Let $I\subseteq {\mathbb R}$ be an open interval,  $\Omega \subseteq {{\mathbb H}^n}$  an open set  and   $\gamma:I\to \Omega$ a  geodesic  segment.  Since  $f:{\cal C}\to {\mathbb R}$  is a differentiable  function  and   $ \gamma'(t)\in T_{\gamma(t)} {{\mathbb H}^n}$ for all $t\in I$, equality \eqref{eq:grad} implies 
\begin{equation} \label{eq:cr1}
\frac{d}{dt}f(\gamma(t)) =\left\langle \grad f(\gamma(t)), \gamma'(t) \right\rangle=
\left\langle {\rm J}\cdot D f(\gamma(t)), \gamma'(t) \right\rangle, \qquad \forall ~ t\in I.
\end{equation}
Moreover, if the function  $f$ is twice differentiable then it holds that
\begin{align} \label{eq:cr2}
\frac{d^2}{dt^2}f(\gamma(t)) &=\left\langle 
\Hess f(\gamma(t))\gamma'(t), \gamma'(t) \right\rangle \notag \\
                                             &=\left\langle {\rm J}\cdot D^2 f(\gamma(t))\gamma'(t), \gamma'(t) \right\rangle + \left\langle {\rm J}\cdot D f(\gamma(t)), \gamma(t) \right\rangle\left\langle \gamma'(t), \gamma'(t) \right\rangle, 
\qquad \forall ~ t\in I.
\end{align}
For each $p, q\in {{\mathbb H}^n}$ the covariant  derivative induces  the linear   isometry relative to the Lorentzian inner product $\langle \cdot , \cdot \rangle$,  $P_{pq} \colon T _{p} {{\mathbb H}^n} \to T _ {q}{{\mathbb H}^n}$ defined by  $P_{pq} v = V(t)$, where $V$ is the unique vector field on the geodesic segment $\gamma:[0,1] \to  {{\mathbb H}^n}$  from $p$ to $q$, i.e., $\gamma(0)=p$ and $\gamma(1)=q$  such that $\nabla V(t) {\gamma'(t)}= \nabla_{\gamma'(t)}V(t) = 0$ and $V(0)=v$, the so-called {\it parallel transport} along  the geodesic segment   $\gamma$ joining  $p$ to $q$. The explicitly formula of $P_{pq}$ is given by 
\begin{equation*}
{P}_{pq}(v):= v-\frac{\langle v,  \log_{q}p \rangle}{{\rm arcosh}^2 (-\langle p , q\rangle)}\left( \log_{q}p+ \log_{p}q \right) .
\end{equation*}
By using \eqref{eq:expinv}, after some algebraic manipulation, the last inequality becomes 
\begin{equation*}
{P}_{pq}(v):=\left[{\rm I}+\frac{1}{1-\langle p, q\rangle}(p+q)q^\top {\rm J}\right]v.
\end{equation*}
Note that  for all geodesic segment  $\gamma:[a,b] \to  {{\mathbb H}^n}$ we have
$ \gamma'(t)={P}_{pq}(\gamma'(a))$, for all $t\in [a, b]$ or equivalently that
$\gamma''(t)=0$,  for all $t\in [a, b]$.  Next we give two standard notations.
We denote the {\it open} and {\it closed ball} of radius $\delta >0$ and center 
$p\in {{\mathbb H}^n}$ by $B_{\delta}(p):=\{q\in {{\mathbb H}^n} : d(p,q)<\delta
\}$  and $\hat{B}_{\delta}(p):=\{q\in {{\mathbb H}^n} : d(p,q)\leq \delta \}$, respectively.

Let us  recall the   {\it  Lorentz group}  $G_{\cal L}$, preserving the   norm \eqref{eq:norm} and metric \eqref{eq:ipef},  defined by 
\begin{equation} \label{eq:lg}
G_{\cal L}:=\left\{Q\in  {\mathbb R}^{(n+1)\times (n+1)}: ~ Q^\top {\rm J} Q={\rm J} \right\}.
\end{equation}
Note  that  $|\det Q|=1$,  for all  $Q\in G_{\cal L}$. Moreover,  $Q^{-1}, Q^\top \in G_{\cal L}$,  for all $Q\in G_{\cal L}$.
\begin{example}
We exhibit two  examples of matrices  $Q\in G_{\cal L}$.  For the first example,
take  $u\in  {\mathbb R}^{n+1}$ such that $\|u\|>0$. Thus, $Q={\rm I}-(2/\|u\|^2
)uu^\top {\rm J}\in G_{\cal L}$. For the second example,  take $u, w\in  {\mathbb R}^{n+1}$ such that $\|u\|=1$ and $\|w\|=1$. Therefore, $Q= {\rm I}+2wu^\top {\rm J}-(1/(1+ u^\top {\rm J}w))(u+w)(u+w)^\top{\rm J}\in G_{\cal L}$.
\end{example} 
We end this section by remarking that  the Lorentz group \eqref{eq:lg} preserves geodesics of  ${{\mathbb H}^n}$. 
\begin{remark} \label{re:lgpg}
First note that for a  given  $Q\in G_{\cal L}$, we have $\|Qv\|=\|v\|$.
Moreover, $p\in {\mathbb H}^{n}$ and $v\in  T_{p}{{\mathbb H}^n}$ if and only if
$Qp\in {\mathbb H}^{n}$ and $Qv\in  T_{p}{{\mathbb H}^n}$, i.e., $v\in
T_p{{\mathbb H}^n}\setminus\{0\}$ if and only if $Qv\in T_{Qp}{{\mathbb
H}^n}\setminus\{0\}$. Consequently,  it follows from \eqref{eq:geoexp} that 
\begin{equation*} 
Q\exp_{p}tv= \exp_{Qp}tQv ,  \qquad  \forall v\in T_p{{\mathbb
H}^n}\setminus\{0\}.
\end{equation*}
Therefore, the Lorentz group preserves  the geodesics of  ${{\mathbb H}^n}$.
\end{remark}
%%%%%%%%%%%%%%%%%%%%%%%%%%%%%%%%%%%%%%%%%%%%%%%
\subsection{Properties of the Intrinsic Distance on the Hyperbolic Space} \label{sec:pid}
%%%%%%%%%%%%%%%%%%%%%%%%%%%%%%%%%%%%%%%%%%%%%%%%
In this section,  we present some important properties of  the intrinsic distance from a  fixed point on the  hyperbolic space. In particular, we present the spectral decomposition of the Hessian of the intrinsic distance. The   { \it intrinsic distance function  on the  hyperbolic space from the fixed  point $ q \in {\mathbb H}^{n}$} 
is the mapping  $d_q:{\mathbb H}^{n} \to {\mathbb R}$ defined by
\begin{equation} \label{eq:dist}
d_q(p):={\rm arcosh} (-\langle p , q\rangle).
\end{equation}
The   intrinsic distance  from $q$, denoted by $d_q$  ,  is twice differentiable 
at $ p \in {\mathbb H}^{n} \backslash  \{q\}$. By combining \eqref{eq:grad} and \eqref{eq:dist},  we can  see that the
{\it  gradient of the distance from $q$} at  $ p$ is given by
\begin{equation} \label{eq:ddf}
 \grad d_q(p):= \frac{-1}{\sqrt{\langle  p, q\rangle^2-1}}\left[{\rm I}+pp^{\top}{\rm J}\right]q,  \qquad  q\neq p.
\end{equation}
Moreover, by using \eqref{eq:Hess} and \eqref{eq:dist}, we obtain that  the {\it  Hessian of the distance from $q$}  at  $ p$  is given  by
\begin{equation} \label{eq:Hessd}
\Hess d_q(p):= \frac{\langle  p, q \rangle}{\sqrt{\langle  p, q\rangle^2-1}} 
\left[{\rm I}+pp^{\top}{\rm J} \right]\left[\frac{1}{\langle  p, q \rangle^2-1}qq^{\top}{\rm J}  -{\rm I}\right],  \qquad  q\neq p.
\end{equation}
Before presenting the  spectral decomposition of the Hessian of  intrinsic distance from a fixed point on 
the  hyperbolic space, we need the following elementary result.
\begin{lemma} \label{le:dinint}
Let $p, q \in {\mathbb H}^{n}$ with  $q\neq p$. Then $\mbox{dim} \left(T_{p}{\mathbb H}^{n}\cap T_{q}{\mathbb H}^{n}\right)=n-1$ and  $\langle q+\langle p, q \rangle p, v\rangle =0$,  for all  $v \in T_{p}{\mathbb H}^{n}\cap T_{q}{\mathbb H}^{n}$. As a consequence, taking  
an orthonormal basis of the subspace $T_{p}{\mathbb H}^{n}\cap T_{q}{\mathbb H}^{n} $, say
$ \{v_1, \ldots, v_{n-1}\} $ 
and defining  $v_n= (q+ \langle p, q \rangle p)/\|q+ \langle p, q \rangle p\|$, 
the set $ \{v_1, \ldots, v_{n-1}, v_n\} $ is an orthonormal basis of $T_{p}{\mathbb H}^{n} $.
\end{lemma}
In the next lemma we present a spectral decomposition of the Hessian of the  intrinsic distance from a 
fixed point on the  hyperbolic space. The results in this lemma and the next one are closely related
to  \cite{ferreira2013newton}, see also its counterparts   on the sphere  in \cite[Lemma 2, Lemma 3]{FerreiraIusemNemeth2014}.
\begin{lemma} \label{le:spd}
Take $ q\in {\mathbb H}^{n} $  and let $ \Hess d_q(p) :T_{p}{\mathbb H}^{n} \to T_{p}{\mathbb H}^{n}$ be the Hessian of 
the intrinsic distance from $ q$ at the point $ p\in {\mathbb H}^{n} \backslash  \{q\} $. Then,
 \begin{equation} \label{eq:eqdis}
 \Hess d_q(p) \, \left( q+ \langle p, q \rangle p\right)=0, \qquad   
\Hess d_q(p)\, v= \frac{-\langle p, q \rangle}{\sqrt{\langle p, q\rangle^2-1}} \,v, \qquad   \forall \, v \in T_{p}{\mathbb H}^{n}\cap T_{q}{\mathbb H}^{n}.
 \end{equation}
 Moreover, $ \lambda_1=0$ and $ \lambda_2=-\langle p, q \rangle/\sqrt{\langle p, q\rangle^2-1}>0$ are the unique eigenvalues 
of $ \Hess d_q(p)$, with algebraic multiplicity $1$  and $n-1$, respectively. 
Moreover,    the Hessian  $ \Hess d_q(p)$  is positive semidefinite.
\end{lemma}
\begin{proof}
 Since $ p\neq  q$, Lemma~\ref{le:fl} implies that   $\langle p, q \rangle \neq -1$ and from \eqref{eq:Hessd}  the Hessian  is well defined.  As  $q^T{\rm J}q=-1$,  simple calculations give 
\begin{equation*}
\left[\frac{1}{\langle  p, q \rangle^2-1}\, qq^T {\rm J} -{\rm I}\right] \left( q+\langle p, q \rangle p\right)=
-\langle p, q \rangle p.
\end{equation*}
On the other hand, $ [{\rm I}+pp^{\top}{\rm J} ](-\langle p, q \rangle p)=0$, which combined with the latter 
equality and \eqref{eq:Hessd}, implies the first equality in \eqref{eq:eqdis}, 
and we also have that $\lambda_1=0$ is an eigenvalue of the Hessian. For proving the 
second equality in \eqref{eq:eqdis}, note that the definitions in \eqref{eq:hs}
imply  that
\begin{equation*}
\langle p, v \rangle=0, \qquad \langle q, v \rangle=0, \qquad  \forall \, v \in T_{p}{\mathbb H}^{n}\cap T_{q}{\mathbb H}^{n}.
\end{equation*}
Thus, the second inequality in \eqref{eq:eqdis} follows from  \eqref{eq:Hessd} and 
the last two equalities. In particular,  the Hessian is a multiple of the identity 
in the subspace $T_{p}{\mathbb H}^{n}\cap T_{q}{\mathbb H}^{n}$. Moreover, due to $ \mbox{dim} T_{p}{\mathbb H}^{n}=n $, we conclude, using Lemma~\ref{le:dinint}, 
that the eigenvalues $ \lambda_1 $ and $\lambda_2$ have  algebraic multiplicity $1$
and $n-1$, respectively, proving the first statement. For proving the second statement, let $ \{v_1, \ldots, v_{n-1}\} $ be an orthonormal
basis of the subspace $T_{p}{\mathbb H}^{n}\cap T_{q}{\mathbb H}^{n} $. Since $\langle p, q \rangle \neq 1$, 
we can  define $v_n= (q+ \langle p, q \rangle p)/\|q+\langle p, q \rangle p\|$. 
Hence,  Lemma~\ref{le:dinint} implies that  $ \{v_1, \ldots, v_{n-1}, v_n\} $ 
is an orthonormal basis of $T_{p}{\mathbb H}^{n} $. Therefore, given $u \in T_{p}{\mathbb H}^{n} $, 
there exist $a_1, \ldots, a_{n-1}, a_n \in {\mathbb R} $ such that $ u=a_1v_1+\cdots +a_{n-1}v_{n-1}+a_nv_n$, 
which, by using the first statement, entails $\langle \Hess d_q(p) u, u\rangle= \lambda_2(a_1^2+\cdots+a_{n-1}^2)$, 
completing the proof of the second statement. 
\end{proof}
Take $ q\in {\mathbb H}^{n} $ and define $ \rho_q:{\mathbb H}^{n} \to {\mathbb R} $ as
\begin{equation} \label{eq:sqdist}
\rho_q(p):=\frac{1}{2}d_q^2(p).
\end{equation}
By using the definition of $ \rho_q$ in \eqref{eq:sqdist} and \eqref{eq:Hess}, it is easy to 
conclude, after some algebra, that
\begin{equation} \label{eq:hsd}
\Hess \rho_q(p)=d_{q}(p)\Hess d_q(p)+[{\rm I}+pp^{\top}{\rm J}]{\rm J}\cdot Dd_q(p)Dd_q(p)^T,
\end{equation}
where $Dd_q(p)$ is the usual derivative of $d_q $ at the point $p$.
\begin{lemma} \label{le:csqd}
Take  $ q\in {\mathbb H}^{n} $ and define $ \Hess \rho_q(p) :T_{p}{\mathbb H}^{n} \to T_{p}{\mathbb H}^{n}$ as the Hessian 
of $\rho_q $ at the point $ p\in {\mathbb H}^{n} \backslash  \{q\} $.   
Then the following equalities hold: 
\begin{equation} \label{eq:chsd}
\Hess \rho_q(p) \, \left( q+ \langle p, q \rangle p\right)= q+ \langle p, q \rangle p, 
\qquad \quad  \Hess \rho_q(p)\, v=
\frac{-\langle p, q \rangle {\rm arcosh} (-\langle p , q\rangle)}{\sqrt{\langle p, q\rangle^2-1}}\,v , 
\end{equation}
for all $v \in T_{p}{\mathbb H}^{n}\cap T_{q}{\mathbb H}^{n} $. As a consequence,   $\mu_1=1$ and  $\mu_2=\langle p, q \rangle {\rm arcosh} (-\langle p , q\rangle)/\sqrt{\langle p, q\rangle^2-1} $ 
are the unique eigenvalues of $ \Hess \rho_q(p)$, with algebraic multiplicity $1$  and $n-1$, 
respectively. Moreover,   the Hessian  $\Hess \rho_q(p)$ is positive definite.
\end{lemma}
\begin{proof}
First note that  taking into account that   $Dd_q(p)=-({\rm J}q)/\sqrt{\langle p, q\rangle^2-1}$, we have
\begin{equation} \label{eq:auxhsd}
{\rm J}\cdot Dd_q(p)Dd_q(p)^T= \frac{1}{\langle p, q\rangle^2-1} qq^T{\rm J}.
\end{equation}
Due to  $ q^T{\rm J}q=-1$, it follows from the last equality that ${\rm J}\cdot Dd_q(p)Dd_q(p)^T(q+\langle p, q \rangle p)=q $. 
On the other hand, $ [{\rm I}+pp^{\top}{\rm J}]q= q+\langle p, q \rangle p$.  Hence, we obtain that
\begin{equation*}
[{\rm I}+pp^{\top}{\rm J}]{\rm J}\cdot Dd_q(p)Dd_q(p)^T (q+ \langle p, q \rangle p)=q+\langle p, q \rangle p.
\end{equation*}
Therefore, combining the last equality, equation \eqref{eq:hsd} and the first 
equality in \eqref{eq:eqdis}, we get that
\begin{equation*}
\Hess \rho_q(p)(q+\langle p, q \rangle p)=q+\langle p, q \rangle p,
\end{equation*}
which is the first equality in \eqref{eq:chsd}. For proving the second one, note first 
that the definition of $T_{q}{\mathbb H}^{n}$ implies that $ q^T{\rm J}v=0 $ for all 
$v \in T_{p}{\mathbb H}^{n}\cap T_{q}{\mathbb H}^{n}$. Then, by using \eqref{eq:auxhsd}, we have
\begin{equation*}
[{\rm I}+pp^{\top}{\rm J}]Dd_q(p)Dd_q(p)^T v =0, \qquad \forall \, v \in T_{p}{\mathbb H}^{n}\cap T_{q}{\mathbb H}^{n}.
\end{equation*}
Hence, equation \eqref{eq:hsd} implies that $\Hess \rho_q(p) v=d_{q}(p)\Hess d_q(p)v $ 
for all $v \in T_{p}{\mathbb H}^{n}\cap T_{q}{\mathbb H}^{n}$. Thus, by using  the second equality in  \eqref{eq:eqdis} 
and the definition of $ d_{q}(p) $  in \eqref{eq:dist}, the second equality 
in \eqref{eq:chsd} follows. The remainder of our proof requires 
similar arguments 
to those in the proof of 
Lemma~\ref{le:spd} (note that in the final part of the proof we must invoke the fact that 
$ {\rm arcosh} (-\langle p , q\rangle) >0 $ and $\langle p, q\rangle<0$, which
holds by applying Lemma~\ref{le:fl} together with  $ p\neq q $). 
\end{proof}
%%%%%%%%%%%%%%%%%%%%%%%%%%%%%%%%%%%%%%%%%%%%%%%%%%%%%%%%%%%%%%%%%%%%%%%%%%%
%%%%%%%%%%%%%%%%%%%%%%%%%%%%%%%%%%%%%%%%%%%%%%%%%%%%%%%%%%%%%%%%%%%%%%%%%%%%%%%%%%%
\section{Convex Sets on the Hyperbolic Space}\label{sec:cs}
%%%%%%%%%%%%%%%%%%%%%%%%%%%%%%%%%%%%%%%%%%%%%%%%%%%%%%%%%%%%%%%%%%%%%%%%%%%%%%%%%%%
In this section we present some properties of the convex sets of the hyperbolic space. It is worth to remark that 
the convex sets on the hyperbolic space $ {{\mathbb H}^n}$ are closely related to   convex cones belonging to   the interior of the   Lorentz cone 
\begin{equation} \label{eq:lc}
{\cal L}:=\left\{x\in {{\mathbb R}^{n+1}}: x_{n+1}\geq \sqrt{x_1^2+\cdots+x_n^2}\right\}.
\end{equation} 
\begin{definition}\label{def:cf}
The set ${\cal C\subseteq} {\mathbb H}^n $ is said to be  \emph{hyperbolically
convex} if for any $p$, $q\in {\cal C}$   the geodesic segment joining
$p$ to $q$  is contained in ${\cal C}$. 

\end{definition}
For each set $ {\cal A} \subseteq {{\mathbb H}^n}$, let ${\cal K_A}$ be the {\it cone spanned by} ${\cal A}$, 
namely, 
\begin{equation} \label{eq:pccone}
{\cal K_A}:=\left\{ tp \, :\, p\in A, \; t\in [0, +\infty) \right\}.
\end{equation}
Clearly, ${\cal K_A}$ is the smallest cone which contains ${\cal A}$ and belongs to the  interior of the   Lorentz cone ${\cal L}$. 
In the next result we relate a hyperbolically convex set with the cone spanned by it.
\begin{proposition} \label{pr:ccs}
The set ${\cal C}\subseteq\mathbb H^n$ is hyperbolically convex if and only if the cone 
${\cal K_C}$ is convex.
\end{proposition}
\begin{proof}
Assume that  ${\cal C\subseteq} {\mathbb H}^n $  is a  hyperbolically convex set. 
Let $x, y \in {\cal K_C} $. For proving that $ {\cal K_C}$ is convex, it suffices
to show that 
\begin{equation} \label{eq:ez}
z=x + y \in {\cal K_C}.
\end{equation}
The definition of $ {\cal K_C}$ implies that there exist  $p, q \in {\cal C}$ and  $ s, t \in [0, +\infty)$ such that  $x=sp $ and $y=tq$. Hence,  due to $z=sp+tq$  with    $p, q \in  {\cal C\subseteq} {{\mathbb H}^n}$,  \eqref{eq:ez} and $\langle p , q\rangle\leq -1$ imply that  $\langle z, z\rangle \leq -(t+s)^2$, which is equivalent to $0<t+s\leq \sqrt{-\langle z,z\rangle} $.   Now we take 
\begin{equation} \label{eq:ccs1}
\gamma_{pq}(t)= \Big( \cosh t + \frac{\langle p, q\rangle \sinh t}{\sqrt{\langle p, q\rangle^2-1}}\Big) p 
+ \frac{\sinh t}{\sqrt{\langle p, q\rangle^2-1}}\;q, \qquad \forall t\in [0, \;d(p,q)].
\end{equation}
the  normalized segment of geodesic from $p$ to $q$.  To proceed  we first  need  to prove  that 
\begin{equation} \label{eq:ficcs}
d\left(p,az\right)\leq d(p,q), \qquad  \quad \gamma_{pq}\left(d(p, az)\right)=az,
\end{equation}
where $a:=1/\sqrt{-\langle z, z\rangle}$. Since the function $[0, +\infty] \ni \tau \mapsto {\rm arcosh}(\tau)$ is  increasing, it follows from \eqref{eq:Intdist}   that  to prove the inequality in \eqref{eq:ficcs} it suffices to show that  $-\langle p,az\rangle \leq -\langle p,q\rangle$ or the equivalent  inequality 
\begin{equation} \label{eq:ficcs1}
0 \leq \langle p,az\rangle-\langle p,q\rangle.
\end{equation}
Due to $z=sp+tq$ and $\langle p, p\rangle=-1$,  direct calculations yield $\langle p,az\rangle-\langle p,q\rangle=a(-s+t\langle p, q\rangle) - \langle p, q\rangle$. Thus,  taking into account that $ \langle p , q\rangle\leq -1$ and $s+t\leq \sqrt{-\langle z,z\rangle}=1/a$ ,  we conclude that
$$
\langle p,az\rangle-\langle p,q\rangle\geq a(s\langle p, q\rangle+t\langle p, q\rangle) - \langle p, q\rangle= a(-\langle p, q\rangle)(1/a-s-t)\geq 0, 
$$
which implies that  \eqref{eq:ficcs1} holds and consequently the inequality in
\eqref{eq:ficcs} also holds. Our next task is to prove the equality in
\eqref{eq:ficcs}. Thus, by using  \eqref{eq:ccs1},  we have to show that
\begin{equation} \label{eq:ccsee}
\gamma_{pq}(d(p, az))= \Big( \cosh d(p, az) + \frac{\langle p, q\rangle \sinh d(p, az)}{\sqrt{\langle p, q\rangle^2-1}}\Big) p 
+ \frac{\sinh d(p, az)}{\sqrt{\langle p, q\rangle^2-1}}\;q=az.
\end{equation}
It follows from \eqref{eq:Intdist} that  $\cosh d(p, az)=-a\langle p, z\rangle$,
which implies that  $\sinh d(p, az)=\sqrt{a^2\langle p, z\rangle^2-1}$. Thus,
considering that $a^2\langle z, z\rangle=-1$ and   $z=sp+tq$  with    $p, q \in
{\cal C\subseteq} {{\mathbb H}^n}$ and  $t \in [0, +\infty)$,    we have
\begin{equation} \label{eq:prie}
 \sinh d(p, az)= \sqrt{a^2\langle p, z\rangle^2-1}= a\sqrt{\langle p, z\rangle^2+ \langle z, z\rangle}=at \sqrt{\langle p, q\rangle^2-1} 
\end{equation}
and  $\cosh d(p, az)=-a\langle p, z\rangle=as-at\langle p,q\rangle$. Hence,
substituting  the last  equality and  \eqref{eq:prie} into    \eqref{eq:ccsee}
and taking into account that  $z=sp+tq$, we obtain that 
\begin{equation*} 
\gamma_{pq}(d(p, az))=\left(as-at\langle p,q\rangle+ \frac{\langle p, q\rangle at\sqrt{\langle p, q\rangle^2-1}}{\sqrt{\langle p, q\rangle^2-1}}\right)p+ \frac{at\sqrt{\langle p, q\rangle^2-1}}{\sqrt{\langle p, q\rangle^2-1}}q=asp+atq=az.
\end{equation*}
which concludes the proof of the equality in \eqref{eq:ficcs}. Since ${\cal C}$ is a  hyperbolically convex set  and  $d(p,az)\leq  d(p,q)$ we obtain that  
$ \gamma_{pq}(d(p, az)) \in {\cal C}$, which together with   \eqref{eq:pccone} and  \eqref{eq:ficcs}   implies that $z=(1/a)\gamma_{pq}(d(p,az)) \in {\cal K_C}$. 
Thus, ${\cal K_C}$ is convex. 

Now,  assume that the cone ${\cal K_C}$ is convex. First note that ${\cal C}={\cal K_C}\cap {{\mathbb H}^n}$. 
Take $p, q \in {\cal C}$  with $q\neq p $.  We must prove that the geodesic 
segment  from $p$ to $q$ is contained in ${\cal C}$.   As  $p, q \in {\cal K_C}$ and $\mathcal{K}_{\cal C}\subseteq {\cal L}$, we conclude that $q\neq - p$. Thus,  $ \langle p, q\rangle < -1 $ and  $d(p,q)>0$.  Let 
\begin{equation*}
 [0, \;d(p,q)]\ni t \mapsto \gamma_{pq}(t)= \alpha (t) p +\beta (t) q
\end{equation*}
be the normalized geodesic segment  from $p$ to $q$, where
\begin{equation*}
\alpha(t):= \cosh t + \frac{\langle p, q\rangle \sinh t}{\sqrt{\langle p, q\rangle^2-1}},  \qquad \beta(t):=\frac{\sinh t}{\sqrt{\langle p, q\rangle^2-1}}.
\end{equation*}
Since $\gamma_{pq}(t) \in {{\mathbb H}^n}$,   $p, q \in {\cal K_C}$ and $ {\cal K_C}$ is a convex cone,  for proving  that $ \gamma_{pq}(t) \in {\cal C}$
for all $t\in [0, \;d(p,q)]$, it suffices to prove that $\alpha(t)\geq 0$ and $\beta(t)\geq 0 $ for all $t\in  [0, \;d(p,q)] $. Due to $\sinh
t\geq 0$ for all $t\geq 0 $ we conclude that $\beta(t)\geq 0 $ for all $t\in  [0, \;d(p,q)] $. We proceed   to prove that $\alpha(t)\geq 0 $ for
all $t\in  [0, \;d(p,q)] $. For that, we first note  that due to   hyperbolic
tangent being an increasing  function, $\cosh d(p, q)=-\langle p, q\rangle$ and
$\sinh d(p, p)=\sqrt{\langle p, q\rangle^2-1}$, we have  
 $$
 \tanh t\leq \tanh d(p,q)=\frac{\sinh d(p, p)}{\cosh d(p, q)}=\frac{\sqrt{\langle p, q\rangle^2-1}}{-\langle p, q\rangle},  \qquad t\in  [0, \;d(p,q)].
 $$
 Hence, taking into account that $\cosh t\geq 0$ for all $t\in {\mathbb R}$ and $ \langle p, q\rangle < -1 $,  we conclude that
 \begin{equation*} 
\alpha(t)= \cosh t \left(1+ \frac{\langle p, q\rangle}{\sqrt{\langle p, q\rangle^2-1}}\tanh t\right)\geq 0,  \qquad t\in  [0, \;d(p,q)], 
\end{equation*}
 which  completes   the proof.
 
\end{proof}
\begin{remark}
	The hyperbolically convex sets are intersections of the hyperboloid  with  convex cones which belong to the interior of ${\cal L}$. 
Indeed, it follows easily from Proposition~\ref{pr:ccs}, that if ${\cal K}\subseteq {\rm int}(\cal L)$ is a convex cone, where ${\cal L}$ is the  Lorentz
cone,  then $ C={\cal K}\cap {{\mathbb H}^n}$ is a hyperbolically convex set and ${\cal K}={\cal K_C}$. 
\end{remark}
\begin{remark} \label{eq:lgpcs}
Let ${\cal C}\subseteq {\mathbb H}^n $ and  $Q\in G_{\cal L}$.  First note that $Q{\cal C}:=\{Qp:~p\in  {\mathbb H}^n \}$. It follows from Remark~\ref{re:lgpg} and Definition~\ref{def:cf} that  ${\cal C}$ is   {hyperbolically convex} if and only if $Q{\cal C}$ is {hyperbolically convex}. 
\end{remark}
%%%%%%%%%%%%%%%%%%%%%%%%%%%%%%%%%%%%%%%%%%%%%%%%%%%%%%%%%%%%%%%%%%%%%%%%%%%%%%%%%%%
\section{Intrinsic Projection Onto Hyperbolically Convex Sets}\label{sec:pj}
%%%%%%%%%%%%%%%%%%%%%%%%%%%%%%%%%%%%%%%%%%%%%%%%%%%%%%%%%%%%%%%%%%%%%%%%%%%%%%%%%%%
In this section we present some properties of the  intrinsic projection onto hyperbolically convex sets on hyperbolic spaces. Let ${\cal C\subseteq} {{\mathbb H}^n} $ be a closed hyperbolically convex set and $p\in {\mathbb H}^n$.  Consider the following constrained  optimization problem 
\begin{align} \label{d:copf}
 \min_{q\in {\cal C}}  d(p,q).
\end{align}
The minimal value of the function ${\cal C}  \ni q \mapsto d(p,q)$ is called the
{\it distance of $p$ from  ${\cal C}$} and it is denoted by $d_ {\cal C}(p)$,
i.e.,  $d_{\cal C} : {{\mathbb H}^n} \to  {\mathbb R} $ is defined by 
\begin{equation*}
 d_{\cal C}(p):=\min_{q\in {\cal C}}  d(p,q).
 \end{equation*}
Since $({{\mathbb H}^n}, d)$  is a complete metric space,  we have the following results.
\begin{proposition} \label{le:cds}
Let $ {\cal C}\in {{\mathbb H}^n}$  be a nonempty subset. Then, $ |d_{\cal C}(p) - d_{\cal C}(q)|\leq d(p, q)$,  for all  $p, q \in {{\mathbb H}^n}$. In particular, the function $d_{\cal C}$ is continuous. 
\end{proposition}
 Note that  due to $ {\cal C}$ being  a closed set and the distance function continuous, the problem \eqref{d:copf} has a solution.  The solution of the problem~\eqref{d:copf} is called  {\it  metric projection}, it was first studied in  \cite{Walter1974}. In the next  proposition  we explicitly give an important property of the  metric projection.
\begin{proposition} \label{pr:conv}
Let ${\cal C\subseteq} {{\mathbb H}^n}$  be a closed hyperbolically convex set and $p\in {{\mathbb H}^n}$.  A point   $y^p\in  {\cal C} $ is a  solution of the  problem~\eqref{d:copf} if and only if 
\begin{equation} \label{eq:fipp}
\left\langle \left({\rm I}+y^p (y^p)^\top{\rm J}\right)p,  \left({\rm I}+y^p (y^p)^\top{\rm J}\right)q\right\rangle \leq 0, \quad  \qquad  ~ \forall  q\in {\cal C}.
\end{equation} 
Furthermore, the solution of  problem~\eqref{d:copf} is  unique.
\end{proposition}
\begin{proof}
First we assume that $y^p$ is a solution of \eqref{d:copf}.  If $p\in {\cal C}$ i.e., $p=y^p$, then the inequality trivially holds.   Assume that
$p\notin {\cal C}$, i.e.,  $p\neq y^p $.  Take $ q\in {\cal C} $ such that $q\neq y^p$ 
and 
\begin{equation} \label{eq:cvxp}
 [0, 1]\ni t \mapsto   \exp_{y^p} (t \,\log_{y^p}q)= \cosh(t d(y^p,q)) y^p+\frac{\sinh(td(y^p,q))}{d(y^p,q)} \,\log_{y^p}q
\end{equation}
be the  geodesic segment  from $y^p $ to $q$. Thus, due to ${\cal C}$ being a  hyperbolically convex set, it follows  
%from  the definition of the projection in  \eqref{d:copf} 
that 
$d(p, y^p)\leq  d(p,\exp_{y^p} (t \,\log_{y^p}q))$, for all $t\in  [0, 1]$. 
Hence, using  \eqref{eq:Intdist}  and \eqref{eq:cvxp}, we conclude that
\begin{equation*}
	{\rm arcosh} (-\langle p , y^p \rangle) \leq  {\rm arcosh}\Big(- \Big\langle p, \cosh(t d(y^p,q)) y^p+\frac{\sinh(td(y^p,q))}{d(y^p,q)} \,\log_{y^p}q\Big\rangle\Big),
\end{equation*}
for all $t\in  [0, 1]$. Since  $1\leq -\langle p, y^p\rangle$, for all $p\in {\cal C}$, and  the function $[1, +\infty]\ni s \mapsto  {\rm arcosh}(s)$ is increasing, we obtain from \eqref{eq:cvxp} that
\begin{equation*}
\Big\langle p \,, \, \cosh(t d(y^p,q)) y^p+\frac{\sinh(td(y^p,q))}{d(y^p,q)} \,\log_{y^p}q \Big\rangle \leq \langle p , y^p\rangle, \qquad \forall \, t\in [0,1].
\end{equation*}
After some algebra,  we conclude from the previous inequality that
\begin{equation*}
\frac{\sinh(t(y^p,q))}{td(y^p, q)} \, \left\langle p, \log_{y^p}q \right\rangle  \leq \frac{1-\cosh(t d(y^p,q))}{td(y^p, q)} \, d(y^p,q) \, \langle p , y^p\rangle, 
\qquad \forall \, t\in [0,1].
\end{equation*}
Letting  $t$ go  to zero in the last inequality we have $ \langle p, \log_{y^p}q \rangle \leq 0$, which, in view of
\eqref{eq:expinv}, yields  
\begin{equation*}
\frac{{\rm arcosh}(- \langle y^p,  q\rangle)}{\sqrt{\langle y^p, q\rangle^2-1 }}  \left\langle p, \left({\rm I}+y^p(y^p)^\top{\rm J}\right)q\right\rangle \leq 0.
\end{equation*}
Thus, due to  $ {\rm arcosh}(- \langle y^p,  q\rangle)>0$,  we have  
$\lng p, \left({\rm I}+y^p(y^p)^\top{\rm J}\right)q\rng\leq 0$. Since   \[\lng
y^p(y^p)^\top{\rm J} p, \left({\rm I}+y^p(y^p)^\top{\rm J}\right)q\rng =0,\] 
the desired inequality \eqref{eq:fipp} follows.  To establish the converse we assume that  $y^p$ satisfies \eqref{eq:fipp}. 
Direct computations show that \eqref{eq:fipp}  is equivalent to
\begin{equation} \label{eq:eqpj}
 \langle p, q\rangle + \langle p, y^p\rangle \langle y^p, q\rangle\leq 0, \qquad \forall \, q\in {\cal C}.
\end{equation}
Since   $  \langle y^p, q\rangle \leq -1$ and    $\langle p, y^p\rangle\leq -1$, we have  $\langle p, y^p\rangle  \langle y^p, q\rangle \geq
-\langle p,y^p\rangle$.  Thus, \eqref{eq:eqpj} implies  that 
\begin{equation*}
1\leq  -\langle p, y^p\rangle\leq -  \langle p, q\rangle, \qquad \forall \, q\in {\cal C}.
\end{equation*}
Due to  the function $[0, +\infty]\ni t \mapsto  {\rm arcosh}(t)$  being  increasing,  the last  inequality implies that 
$$
 {\rm arcosh}(- \langle p, y^p\rangle)\leq {\rm arcosh} (-\langle p, q\rangle), \qquad \forall q \in {\cal C}, 
 $$ 
or equivalently that $d(p,y^p)\leq d(p,q)$,   for all $q \in {\cal C} $. Therefore,
$y^p$ is a solution of \eqref{d:copf} and the converse is
proved. For the uniqueness, let  $p, \hat p\in {\cal P_{\cal C}}(p)$. Since $ y^p, \hat y^p\in {\cal C}$ and $\langle y^p, \hat y^p\rangle\leq -1$ (see Lemma~\ref{le:fl}) , by the first statement, it follows 
from the equivalence of \eqref{eq:fipp} and \eqref{eq:eqpj} that
\begin{equation*}
 \langle p,  \hat y^p\rangle \leq - \langle p,\hat  y^p\rangle \langle y^p,  \hat y^p\rangle=\langle p,\hat  y^p\rangle|\langle y^p,  \hat y^p\rangle|, 
 \end{equation*}
which implies that   $ \langle p, \hat y^p\rangle \leq  \langle p, \hat y^p\rangle \langle y^p, \hat y^p\rangle^2$.  
Due to  $\langle p, \hat y^p\rangle\leq -1$, we obtain that $1\geq \langle \hat
y^p, \hat y^p\rangle^2  $.  Hence, taking into account  that $\langle \hat y^p,
\hat y^p\rangle\leq -1$, we conclude that   $\langle y^p, \hat y^p\rangle= -1$. Therefore, from Lemma~\ref{le:fl}  we conclude that $y^p= \hat y^p$ and  the solution set of the problem~\eqref{d:copf}   is a singleton set, which concludes the proof. 

\end{proof}
It follows from Proposition~\ref{pr:conv} that the {\it projection mapping}  ${\cal P_{\cal C}}:  {{\mathbb H}^n} \to {\cal C}$ given by 
\begin{align} \label{d:prjection}
{\cal P_{\cal C}}(p):= \arg\min_{q\in {\cal C}}  d(p,q) 
\end{align}
is well defined. Moreover,  \eqref{d:prjection}  is equivalent to the following inequality
\begin{equation} \label{d:prjectionef}
\left\langle \left({\rm I}+{\cal P_{\cal C}}(p){\cal P_{\cal C}}(p)^T{\rm J}\right)p,  \left({\rm I}+{\cal P_{\cal C}}(p){\cal P_{\cal C}}(p)^T{\rm J}\right)q\right\rangle \leq 0, \quad  \qquad  ~ \forall  q\in {\cal C}, ~ \forall  p\in  {{\mathbb H}^n}.
\end{equation} 
Considering that Lemma~\ref{le:fl} implies that for all $p, q \in {{\mathbb H}^n}$ we have $\langle {\cal P_{\cal C}}(p), q\rangle\leq  -1$,  we conclude from  \eqref{eq:expinv} that \eqref{d:prjectionef}  can be equivalently  stated as follows 
\begin{equation}  \label{d:uexp}
\langle \log_{{\cal P_{\cal C}}(p)}p, \log_{{\cal P_{\cal C}}(p)}q\rangle \leq 0,  \quad \qquad  ~ \forall  q\in {\cal C}, ~ \forall  p\in  {{\mathbb H}^n},
\end{equation} 
see \cite[Corollary 3.1]{FerreiraOliveira2002}. Furthermore,  since that   the
function $[0, +\infty] \ni \tau \mapsto {\rm arcosh}(\tau)$ is  increasing, it
follows from \eqref{eq:Intdist} that  \eqref{d:prjection}, \eqref{d:prjectionef}
and \eqref{d:uexp} are  also  equivalent to   
\begin{align} \label{d:prjcef}
{\cal P_{\cal C}}(p):= {\rm argmin}_{q\in {\cal C}}  (-\langle p, q\rangle).
\end{align}
An immediate consequence  of \eqref{d:prjcef}  is the  montonicity of the projection
mapping, stated as follows:
\begin{proposition}
Let ${\cal C\subseteq} {{\mathbb H}^n}$ be a nonempty closed hyperbolically convex set. Then 
\begin{equation*}
\langle {\cal P_{\cal C}}(p) - {\cal P_{\cal C}}(q), p-q \rangle \geq 0, \qquad      \forall  ~ p, q\in {\cal C}.
\end{equation*}
\end{proposition}
\begin{proof}
Take $ p, q \in {{\mathbb H}^n}$. Since ${\cal P_{\cal C}}(p),  \ {\cal P_{\cal C}}(q) \in {\cal C} $, it follows from  \eqref{d:prjcef} that    $ -\langle p, {\cal P_{\cal C}}(p) \rangle \leq  - \langle p, {\cal P_{\cal C}}(q) \rangle $ and  $-\langle q, {\cal P_{\cal C}}(q) \rangle \leq -\langle q,  {\cal P_{\cal C}}(p) \rangle $.  Hence,      $  \langle p,  {\cal P_{\cal C}}(p)-{\cal P_{\cal C}}(q) \rangle \geq 0 $ and $\langle -q,  {\cal P_{\cal C}}(p)-{\cal P_{\cal C}}(q) \rangle \geq 0$. Therefore, summing  the last two inequalities the desired  inequality follows.

\end{proof}
\begin{proposition} 
Let ${\cal C\subseteq} {{\mathbb H}^n}$ be a nonempty closed hyperbolically convex set. Then $ P_{\cal C}$ is continuous. 
\end{proposition}
\begin{proof}
Let $ \{p^k\}\subseteq  {{\mathbb H}^n}$ be such that $ \lim_{k\to +\infty}p^k=p$.  Since Proposition~\ref{le:cds} implies that  $d_{\cal C}$ is continuous and taking into account that  \eqref{d:prjection}  implies 
\begin{equation} \label{eq:epcc}
d\lf(p^{k}, P_{\cal C}\lf(p^{k}\rg)\rg)=d_{\cal C}\lf(p^k\rg), 
\end{equation} 
we conclude that $\lf(d\lf(p^{k}, P_{\cal C}\lf(p^{k}\rg)\rg)\rg)_{k\in {\mathbb N}}$ is a bounded sequence. Consequently,  considering that 
$$
d\lf(p, P_{\cal C}\lf(p^{k}\rg)\rg)\leq d\lf(p, p^{k}\rg) + d\lf(p^{k}, P_{\cal
C}\lf(p^{k}\rg)\rg), 
$$
we also have  that  $\lf(P_{\cal C}\lf(p^{k}\rg)\rg)_{k\in {\mathbb N}}$ is also
bounded.  Let $q \in {\cal C}$  be a cluster point of  $\lf(P_{\cal
C}\lf(p^{k}\rg)\rg)_{k\in {\mathbb N}}$ and let $\lf(p^{k_j}\rg)_{j\in {\mathbb
N}}$be such that  $ \lim_{j\to +\infty}P_{\cal C}\lf(p^{k_j}\rg)=q$.   Hence,
using \eqref{eq:epcc} we have $d_{\cal C}\lf(p^{k_j}\rg)= d\lf(p^{k_j}, P_{\cal
C}\lf(p^{k_j}\rg)\rg)$, for all $j\in {\mathbb N}$.  Thus,  letting $j$ goes to $+\infty$ and using Proposition~\ref{le:cds}  we have $ d_{\cal C}(p)= d(p, q)$, which due to the second part of Proposition~\ref{pr:conv} implies that $q=P_{\cal C}(p)$. Consequently,    
$\lf(P_{\cal C}\lf(p^{k}\rg)\rg)_{k\in {\mathbb N}}$ has only one cluster point,
namely, $P_{\cal C}(p)$.  Thus, $ \lim_{k\to +\infty}P_{\cal C}\lf(p^{k}\rg)=P_{\cal C}(p)$ and the proof is  concluded.

\end{proof}
%%%%%%%%%%%%%%%%%%%%%%%%%%%%%%%%%%%%%%%%%%%%%%%%%%%%%%%%%%%%%%%%%%%%%%%%%%%%%%%%%%%
\section{Hyperbolically Convex Functions}\label{sec:cf}
%%%%%%%%%%%%%%%%%%%%%%%%%%%%%%%%%%%%%%%%%%%%%%%%%%%%%%%%%%%%%%%%%%%%%%%%%%%%%%%%%%%
In this section we study the basic properties of convex functions on the hyperbolic space. In particular, for differentiable convex functions, 
  the  first and  second order characterizations  will be presented.

\begin{definition}\label{def:cf-b}
Let $ {\cal C}\subseteq {{\mathbb H}^n}$ be a hyperbolically convex set and $I\subseteq {\mathbb R}$ an interval.  
A function $f:{\cal C}\to {\mathbb R}$  is said to be hyperbolically convex
(respectively, strictly hyperbolically convex) if for any   geodesic segment $\gamma:I\to {\cal C}$, the composition 
$ f\circ \gamma :I\to {\mathbb R}$ is convex (respectively, strictly convex) in the usual sense. 
\end{definition}
In the following remark we state  some general properties of  hyperbolically convex,  which follow directly from Definition~\ref{def:cf-b}.
\begin{remark} \label{re:pohcf}
It follows from Definition~\ref{def:cf-b} that $f:{\cal C}\to {\mathbb R}$  is a  hyperbolically convex function 
if and only if the {\it epigraph} $\mbox{epi}f:=\left\{(p, \mu):~ p\in {\cal C},
\, \mu \in {\mathbb R}, \, f(p)\leq  \mu   \right\}$, is convex in ${{\mathbb
H}^n}\times {\mathbb R}$. Moreover,  if $f:{\cal C}\to {\mathbb R}$ is a
hyperbolically convex function, then   the sub-level sets $\{p\in {\cal C}:\;
f(p)\leq a\}$ are   hyperbolically convex sets,  for all $a\in {\mathbb R}$.
Furthermore,  if  $f, f_1, \cdots f_n:  {{\mathbb H}^n}  \to {\mathbb R}$ are
hyperbolically convex in ${\cal C}$, then  $\zeta f$  and $f_1+\cdots+ f_n$ are hyperbolically convex in ${\cal C}$, for all $\zeta\geq 0$.
\end{remark}
The next  proposition follows from Remark~\ref{re:lgpg} and Remark~\ref{eq:lgpcs},  Definition~\ref{def:cf} and Definition~\ref{def:cf-b}. 
\begin{proposition} \label{pr:lgpcf}
Let $ {\cal C}\subseteq {{\mathbb H}^n}$ be a hyperbolically convex set,  $Q\in G_{\cal L}$ and ${\cal D}:=\{Q^{-1}p:~p\in {\cal C} \}$. The  function $f:{\cal C}\to {\mathbb R}$  is   hyperbolically convex if and only if  $f\circ Q:{\cal D}\to {\mathbb R}$   defined by $f\circ Q (q):=f(Qq)$  is   hyperbolically convex.
\end{proposition}
%%%%%%%%%%%%%%%%%%%%%%%%%%%%%%%%%%%%%%%%%%%%%%%%%%%%%%%%%%%%%%%%%%%%%%%%%%%%%%%%%%%
\subsection{Characterization  of Hyperbolically Convex Functions}
%%%%%%%%%%%%%%%%%%%%%%%%%%%%%%%%%%%%%%%%%%%%%%%%%%%%%%%%%%%%%%%%%%%%%%%%%%%%%%%%%%%
In this section we present  first and second order characterization  for  hyperbolically convex functions on hyperbolic spaces.
\begin{proposition} \label{pr:cfocf}
Let ${\cal C} \subseteq {{\mathbb H}^n}$ be an open  hyperbolically convex set and $f:{\cal C}\to {\mathbb R}$ be a differentiable function. 
The function $f$ is  hyperbolically convex if and only if    $f(q)\geq f(p)+
\langle \grad f(p),  \log_{p}q \rangle$,  for all  $ p,  q\in {\cal
C}$ and $q\neq p$,  or equivalently, 
$$
f(q)\geq f(p)+ 
\frac{{\rm arcosh} (-\langle p , q\rangle)}{\sqrt{1-\langle p, q\rangle ^2}} \left\langle  [{\rm I}+pp^{\top}{\rm J}]{\rm J}\cdot Df(p) \, , \,q \right\rangle, 
\qquad \forall \, p,  q\in {\cal C}, \; q\neq p,
$$
where $Df$ is the usual gradient of $f$.
\end{proposition}
\begin{proof}
By using \eqref{eq:cr1},  the usual characterization of scalar convex functions implies that, 
for all minimal geodesic  segment $\gamma:I\to {\cal C}$, the composition $ f\circ \gamma :I\to {\mathbb R}$ 
is convex if and  only if
$$
f(\gamma(t_2))\geq  f(\gamma(t_1))+ 
\left\langle {\rm J}\cdot D f(\gamma(t_1)), \gamma'(t_1) \right\rangle (t_2-t_1), \qquad  \forall ~ t_2, t_1 \in I.
$$
Note that if  $\gamma:[0,1] \to {\cal C}$ is the   geodesic segment from $p=\gamma(0)$ to $q=\gamma(1)$, 
then it  may be represented as $\gamma(t)=\exp_{p}t\log_{p}q$. Moreover,  $\gamma'(0)=\log_{p}q$ and  $\gamma'(1)=-\log_{q}p$. Therefore,  the first inequality of the proposition is an immediate consequence of  
the inequality above,  Definition~\ref{def:cf-b} and equation \eqref{eq:expinv}. For concluding the proof, note that 
equations  \eqref{eq:grad} and \eqref{eq:expinv} together with Remark~\ref{eq:sadj} imply the equivalence between the two inequalities of the lemma.

\end{proof}
\begin{proposition}
Let ${\cal C} \subseteq {{\mathbb H}^n}$ be an open  hyperbolically convex set and  $f:{\cal C} \to {\mathbb R}$ a differentiable function. 
The function $f$ is  hyperbolically convex if and only if   $  \grad f$  satisfies the inequality
$\langle \grad f(p),  \log_{p}q \rangle + \langle \grad f(q),  \log_{q}p \rangle \leq 0$, for all  $p,  q\in {\cal C}$ and $q\neq p$, 
or equivalently, 
$$
 \left\langle {\rm J}\cdot Df(p)-  {\rm J}\cdot Df(q) , p-q\right\rangle -  (\langle p,  q \rangle +1)\left[ \langle {\rm J}\cdot Df(p), p \rangle+  \langle {\rm J}\cdot Df(q), q \rangle\right] \geq 0, \qquad  \forall \, p,  q\in {\cal C}, \; q\neq p,
$$
where $Df$ is the usual gradient of $f$. 
\end{proposition}
\begin{proof}
Using \eqref{eq:cr1},  the usual first order characterization of convex functions implies that, 
for all minimal geodesic  segments $\gamma:I\to {\cal C}$, the composition $ f\circ\gamma :I\to {\mathbb R}$ 
is convex if and  only if
$$
\left[ \left\langle {\rm J}\cdot D f(\gamma(t_2)), \gamma'(t_2) \right\rangle 
-\left\langle {\rm J}\cdot D f(\gamma(t_1)), \gamma'(t_1) \right\rangle  \right] (t_2-t_1)
\geq 0, \qquad \forall ~ t_2, t_1 \in I.
$$
Note that if  $\gamma:[0,1] \to {\cal C}$ is the segment  of   geodesic from $p=\gamma(0)$ to $q=\gamma(1)$,
then it  may be represented as $\gamma(t)=\exp_{p}t\log_{p}q$.  Moreover,  $\gamma'(0)=\log_{p}q$ and  $\gamma'(1)=-\log_{q}p$.
Therefore,   the first inequality of the proposition follows by combining  the previous inequality with  
Definition~\ref{def:cf-b} and \eqref{eq:expinv}.  
For concluding the proof, note that equations  \eqref{eq:grad} and \eqref{eq:expinv}   imply the
equivalence between the two inequalities of the lemma.

\end{proof}
\begin{proposition} \label{pr:sgmd2}
Let ${\cal C} \subseteq {{\mathbb H}^n}$ be an open hyperbolically convex set and  $f:{\cal C} \to {\mathbb R}$ be  a twice differentiable function. 
The function $f$ is hyperbolically convex if and only if the  Hessian $  \Hess f$ on the hyperbolic space   
satisfies the inequality $\left\langle \Hess f(p) v, v \right\rangle \geq 0$, for all  $p\in {\cal C}$  and all \, $v\in T_p{{\mathbb H}^n}$, 
or equivalently, 
$$
\left\langle {\rm J}\cdot D^2f(p)v, v \right\rangle + \langle {\rm J}\cdot Df(p), p\rangle\left\langle v, v \right\rangle\geq 0, 
\qquad \forall\, p\in {\cal C},~  \forall \, v\in T_p{{\mathbb H}^n},
$$
where $ D^2f(p)$ is the usual Hessian  and $Df(p)$ is the usual gradient of $f$ 
at a point $p\in {\cal C}$. If the above inequalities  are strict then $f$ is strictly hyperbolically convex. 
\end{proposition}
\begin{proof}
By using \eqref{eq:cr2}, the usual second order characterization of hyperbolically convex functions implies that, 
for all minimal geodesic  segment $\gamma:I\to {\cal C}$, the composition $ f\circ \gamma :I\to {\mathbb R}$ 
is convex if and  only if
$$
\left\langle {\rm J}\cdot D^2 f(\gamma(t))\gamma'(t), \gamma'(t) \right\rangle +  \left\langle {\rm J}\cdot D f(\gamma(t)), \gamma(t) \right\rangle\left\langle \gamma'(t), \gamma'(t) \right\rangle 
\geq 0, \qquad \forall ~ t\in I.
$$
If the last inequality is strict then $ f\circ \gamma$ is strictly convex. Therefore,
the result follows by combining  the above inequality with  Definition~\ref{def:cf-b}.  For concluding the proof, 
note that equation \eqref{eq:Hess} together with Remark~\ref{eq:sadj}  imply the equivalence between the two inequalities of the lemma.

\end{proof}
\begin{example} 
Fix $q\in {{\mathbb H}^n}$.  The function $d_{q}(\cdot) :{{\mathbb H}^n} \to {\mathbb R} $ is hyperbolically convex.
In general, taking a hyperbolically convex set ${\cal C}\subseteq  {{\mathbb H}^n}$, 
the function  $d_{q}(\cdot)  : {\cal C}\to {\mathbb R}$ is hyperbolically convex. 
Indeed,  the  hyperbolic convexity of  
$d_{q}(\cdot) $ follows by combining Lemma~\ref{le:spd} with  Proposition~\ref{pr:sgmd2}.
\end{example}
\begin{example}  
Fix $q \in {{\mathbb H}^n}$.  The function   $ \rho_q: {{\mathbb H}^n}: \to {\mathbb R} $ defined as $\rho_q(p):=\frac{1}{2}d_q^2(p)$
 is strictly hyperbolically convex.  In general, taking a hyperbolically convex set ${\cal C}\subseteq {{\mathbb H}^n}$, 
the function  $\rho_q  : {\cal C}\to {\mathbb R}$ is strictly hyperbolically convex. 
Indeed,    the result follows by combining Lemma~\ref{le:csqd} with  Proposition~\ref{pr:sgmd2}.
\end{example}
\begin{example} 
Take $\tilde{p}=(0,\cdots,0,1)\in {\mathbb R}^{n+1}$ and  the hyperbolically convex set 
${\cal C}=\{p\in {{\mathbb H}^n} :~  p^{1}> 0, \ldots, p^{n}> 0  \} $.   The function 
$\psi : {\cal C} \to {\mathbb R}$  defined by 
$\psi(p)=-\ln \left(-1-\langle \tilde{p}, p\rangle \right)$  is hyperbolically convex.  
Indeed,   considering that $Df(p)=-(1+\langle \tilde{p}, p\rangle)^{-1}{\rm J}{\tilde p}$ and
$D^2f(p)=(1+\langle \tilde{p}, p\rangle)^{-2}{\rm J}{\tilde p}{\tilde p}^\top{\rm J}$, the 
hyperbolical  convexity of  $\psi $ follows by combining  Lemma~\ref{le:fl} and  
Proposition~\ref{pr:sgmd2}.
\end{example}
%%%%%%%%%%%%%%%%%%%%%%%%%%%%%%%%%%%%%%%%%%%%%%%%%%%%%%%%%%%%%%%%
\subsection{Hyperbolically Convex Quadratic Functions} \label{sec:hcqf}
%%%%%%%%%%%%%%%%%%%%%%%%%%%%%%%%%%%%%%%%%%%%%%%%%%%%%%%%%%%%%%%%
In this section we study the hyperbolic convexity of the quadratic function  $f(p)=p^\top Ap$, for  $A=A^\top\in{\mathbb R}^{(n+1)\times(n+1)}$.  We begin with a general characterization.
\begin{corollary}\label{cor:lhc}
	Let $A=A^\top\in{\mathbb R}^{(n+1)\times(n+1)}$ and  $f:{\mathbb H}^n\to {\mathbb R}$ defined by
	$f(p)=p^\top Ap$. The function $f$ is hyperbolically convex if and 
	only if \[v^\top A v+p^\top A p\geq 0,\qquad\forall p,v\in{\mathbb R}^{n+1}
	\quad\mbox{with}\quad p^\top Jp=-1,
	\quad v^\top Jv=1,\quad p^\top Jv=0.\] 
\end{corollary}
\begin{proof}
Considering that $Df(p)=2Ap$,  $D^2f(p)=2A$ and  $ {\rm J} {\rm J}= {\rm I}$,  we conclude that 
$$
\left\langle {\rm J}\cdot D^2f(p)v, v \right\rangle + \langle {\rm J}\cdot Df(p), p\rangle\left\langle v, v \right\rangle=2 v^\top Av+ 2p^\top Ap\left\langle v, v \right\rangle.
$$
Thus,  it follows from Proposition~\ref{pr:sgmd2} that $f$ is hyperbolically convex in ${{\mathbb H}^n}$ if and only if
 $v^\top Av+ p^\top Ap \geq 0$, for all $p\in {{\mathbb H}^n}$,  all $v\in T_p{{\mathbb H}^n}$ such that  $v^\top Jv=1$. Considering that $v\in T_p{{\mathbb H}^n}$ with $p\in {{\mathbb H}^n}$  if and only if $v^\top Jv=1$ and $p^\top Jv=0$ with $p\in {{\mathbb H}^n}$, the result follows.
 
\end{proof}
Next, we use  the Lorentz group \eqref{eq:lg} to present some examples of
hyperbolically  convex quadratic functions. Before that, we need the following result. 
\begin{corollary}\label{cr:lhc}
Let $A=A^\top\in{\mathbb R}^{(n+1)\times(n+1)}$ and $Q\in G_{\cal L}$.  Then,  $f:{\mathbb H}^n\to {\mathbb R}$ defined by $f(p)=p^\top Ap$ is   hyperbolically convex if and only if $g:{\mathbb H}^n\to {\mathbb R}$ defined by  $g(q)= q^\top (Q^\top A Q)q $ is   hyperbolically convex.
\end{corollary}
\begin{proof}
 Since $g(q)=f(Qq)$, the equivalence  follows from Proposition~\ref{pr:lgpcf}.
 
\end{proof}
As an application of Corollaries ~\ref{cor:lhc}  and ~\ref{cr:lhc}  in the
following we present an example of a hyperbolically  convex quadratic function.
\begin{example}
Take a diagonal matrix $D\in {\mathbb R}^{(n+1)\times (n+1)}$ denoted by $D={\rm diag} (d_1, \ldots, d_n, d_{n+1})$.  Assume that $ d_{min} + d_{n+1}\geq 0$,  where $d_{min}:=\min\{d_1, \ldots, d_n\}$. Then,  for each  $Q\in G_{\cal L}$,  the  function $g:{\mathbb H}^n\to{\mathbb R}$ defined by $g(p)=p^\top Q^\top DQ p$ is   hyperbolically convex. Indeed,  take $q, u\in {\mathbb R}^{n+1}$ such that  $q^\top Jq=-1$, $u^\top Ju=1$  and $p^\top Jv=0$. Thus,  we have  $q_{n+1}^2=\sum_{i=1}^{n}q_i^2+1$ and $u_{n+1}^2=\sum_{i=1}^{n}u_i^2-1$. Hence,  since  $ d_{min} + d_{n+1}\geq 0$, we  obtain that 
\begin{equation*} 
 u^\top Du+ q^\top Dq=\sum_{i=1}^{n}(d_i+d_{n+1})u_i^2+\sum_{i=1}^{n}(d_i+d_{n+1})q_i^2\geq  (d_{min} + d_{n+1})\Big(\sum_{i=1}^{n}u_i^2+\sum_{i=1}^{n}q_i^2\Big)\geq 0.
\end{equation*}
Thus,  Corollary~\ref{cor:lhc} implies that  $f:{\mathbb H}^n\to{\mathbb R}$ defined by $f(p)=p^\top Dp$ is   hyperbolically convex. Therefore, applying Corollary~\ref{cr:lhc} we conclude that $g$ is   hyperbolically convex.
\end{example}
To continue  with our study of  hyperbolic convexity of  quadratic functions,  we
denote  the boundary of the Lorentz cone  \eqref{eq:lc} by 
$$
\partial{\cal L}:=\left\{x\in {\cal L}:~ x^\top Jx=0\right\}.
$$
In order to simplify notations, for a given $x\in{{\mathbb R}^{n+1}} $,  we
consider the following decomposition:
\begin{equation} \label{eq:dsv}
x=({\bar x}^\top, x_{n+1}) \in {{\mathbb R}^{n+1}},\qquad  {\bar x}:=\left(x_1,\dots,x_n\right)^\top\in{\mathbb R}^n,\quad  x_{n+1}\in{\mathbb R}.
\end{equation}
\begin{lemma} \label{le:ebc}
Let  $x,y\in \partial{\cal L}$. The following three statements are equivalent:
\begin{enumerate}
	\item[(i)] $x^\top Jy\ne 0$; 
         \item[(ii)] $y\ne\alpha x$, for all $\alpha\in {\mathbb R}$; 
         \item[(iii)] $x^\top Jy< 0$.
\end{enumerate}	
\end{lemma}
\begin{proof}
First, we  prove (i) is equivalent to (ii). Assume  (i) holds.   If there exists
$\alpha\in{\mathbb R}$ such that $y=\alpha x$, then $x^\top Jy=x^\top J(\alpha
x)=\alpha x^\top Jx=0$, which contradicts $x^\top Jy\ne 0$. Hence, $y\ne\alpha
x$, for all $\alpha\in{\mathbb R}$, and (ii)  holds.  For the converse, assume
(ii) holds.  By contradiction, assume that  $x^\top Jy=0$.  Since  $y,
z\in\partial{\cal L}$, by using the notation introduced in \eqref{eq:dsv},  we
have  $y^{n+1}= \sqrt{{\bar y}^\top {\bar y}}$ and $z_{n+1}= \sqrt{{\bar z}^\top
{\bar z}}$. Thus,  due to $y^Tz=0$, we have ${\bar y}^T{\bar z}=- y^{n+1}
z_{n+1}$, or equivalently
$$
{\bar y}^\top {\bar z}=  -\sqrt{{\bar y}^\top {\bar y}}  \sqrt{{\bar z}^\top {\bar z}}.
$$
Hence, Cauchy's inequality implies that  there exists  a $\alpha\ge 0$ such that  ${\bar y}=-\alpha {\bar z}$. Furthermore,  
$$
y^{n+1}= \sqrt{{\bar y}^\top {\bar y}} = \alpha\sqrt{{\bar z}^\top {\bar z}} , 
$$
 which gives  $y^{n+1}=\alpha z_{n+1}$. Thus, we conclude that $y=-\alpha Jz$, which implies $y=\alpha x$ and we have a contradiction. Therefore,  $x^\top Jy\neq 0$ and  (i) holds. 

Now, we  prove (i) is equivalent to (iii). Assume  (i) holds.  Since ${\cal L}$
is a closed and  convex cone,  and  $x,y\in \partial{\cal L}\subseteq  {\cal L}
$, we have $x+y \in  {\cal L}$.  Thus, 
$$
0\geq (x+y)^\top J (x+y)= x^\top J x+2x^\top Jy+ y^\top Jy=2x^\top Jy.
$$
 Hence, $x^\top Jy\ne 0$ implies  $x^\top Jy< 0$. Therefore, the item  (iii) holds. Conversely, (iii) implies (i) is immediate,  which concludes the proof. 
 
\end{proof}
If we make some transformations in Corollary~\ref{cor:lhc} , we will obtain the following result.
\begin{lemma} \label{le:escc}
Let $A=A^\top\in{\mathbb R}^{(n+1)\times(n+1)}$. The following three conditions are equivalent:
\begin{enumerate}
	\item[(i)] The  function $f:{\mathbb H}^n\to {\mathbb R}$ defined by $f(p)=p^\top Ap$  is hyperbolically convex; 
         \item[(ii)]  $x^\top A x+y^\top A y\geq 0$, for all $x,y\in{\mathbb R}^{n+1}$  with $x,y\in \partial{\cal L}$ and $x^\top Jy=-1$; 
         \item[(iii)]  $z^\top A z+w^\top A w\geq 0$, for all $z,w\in{\mathbb R}^{n+1}$  with $z,w\in \partial{\cal L}$ and $z^\top Jw<0$.	
\end{enumerate}
\end{lemma}
\begin{proof}
First we prove the equivalence  between (i) and (ii). For that   it is convenient to consider the following  invertible transformations 
\begin{equation} \label{eq:cofv}
p=\frac{1}{\sqrt{2}}(x+y),\quad v=\frac{1}{\sqrt{2}}(x-y)  \quad \mbox{if and  only if} \quad   x=\frac{1}{\sqrt{2}}(p+v), \quad y=\frac{1}{\sqrt{2}}(p-v), 
\end{equation}
where $x,y, p, v\in{\mathbb R}^{n+1}$.  By using the first two equalities in \eqref{eq:cofv},  after  some calculations, we have 
\begin{align}
  2p^\top Jp&=x^\top Jx+2x^\top Jy+y^\top Jy,\notag\\
2v^\top Jv&=x^\top Jx-2x^\top Jy+y^\top Jy,  \label{eq:cofvc1} \\
2p^\top Jv&=x^\top Jx-y^\top Jy.\notag
\end{align}
On the other hand,   by using the last two inequalities in   \eqref{eq:cofv} we obtain the following three  equalities  
\begin{align} 
  2x^\top Jx&=p^\top Jp+2p^\top Jv+v^\top Jv,\notag \\
 2y^\top Jy&=p^\top Jp-2p^\top Jv+v^\top Jv, \label{eq:cofvc2} \\
  2x^\top Jy&=p^\top Jp-v^\top Jv. \notag
\end{align}
Moreover, the equalities in \eqref{eq:cofv} also  imply that 
\begin{equation} \label{eq:cofvc3}
v^\top A v+p^\top Ap=x^\top A x+y^\top A y.
\end{equation}
First we prove (i) implies (ii). Take $x,y\in \partial{\cal L}$  and $x^\top Jy=-1$, and consider the transformation \eqref{eq:cofv}. Thus, by using \eqref{eq:cofvc1}, we conclude that  $p^\top Jp=-1$, $v^\top Jv=1$ and  $p^\top Jv=0$.  Hence, item~(i) together with Corollary~\ref{cor:lhc} implies that $v^\top A v+p^\top Ap\geq 0$. Therefore, by using \eqref{eq:cofvc3}, we conclude that $x^\top A x+y^\top A y\geq 0$ and  item~(ii) holds.

Next we prove that (ii) implies (i). Assume that the  item~(ii) holds, and  take $p,v\in{\mathbb R}^{n+1}$  with $p^\top Jp=-1$, $v^\top Jv=1$ and  $p^\top Jv=0$,  and consider  \eqref{eq:cofv}. Hence, by using \eqref{eq:cofvc2} we have $x ~(\mbox{or}-x)\in\partial{\cal L}$, $y~(\mbox{or}-y)\in \partial{\cal L}$  and $x^\top Jy=-1$, and item~(ii) implies that  $x^\top A x+y^\top A y\geq 0$. Thus, \eqref{eq:cofvc3} implies that $v^\top A v+p^\top Ap\geq 0$, which implies that   item~(i) holds. 

We proceed to prove the equivalence  between (ii) and (iii). Assume that item~(ii) holds and take  $z,w\in \partial{\cal L}$ and $z^\top Jw<0$.  Since $z^\top Jw<0$ we define 
\begin{equation} \label{eq;sie}
x=\frac{z}{\sqrt{-z^\top Jw}},\qquad y=\frac{w}{\sqrt{-z^\top Jw}}.
\end{equation} 
Thus, considering that  $z,w\in \partial{\cal L}$ and $z^\top Jw<0$, some calculations show that $x,y\in \partial{\cal L}$ and $x^\top Jy=-1$. Therefore, using  \eqref{eq;sie}  together  item~(ii), we conclude that
$$
z^\top A z+w^\top A w= -z^\top Jw\left(x^\top A x+y^\top A y\right)\geq 0, 
$$
and item~(iii) holds. Finally,  (iii) implies  (ii) is immediate, which  concludes the proof. 

\end{proof} 
In the following theorem we present a characterization  for hyperbolically convex
quadratic functions in term of the matrix defining it. In particular,  we show
that the study of hyperbolically convex quadratic functions reduces to the study
of their behavior on  the boundary of the Lorentz cone.
\begin{theorem}\label{th:pc}
Let $A=A^\top\in{\mathbb R}^{(n+1)\times(n+1)}$. The following four  conditions are equivalent:
\begin{enumerate}
	\item[(i)] The quadratic  function   $f: {{\mathbb H}^n} \to {\mathbb R} $ defined by $f(p):=p^\top Ap$  is  hyperbolically convex; 
	\item[(ii)] The matrix $A$ is $\partial{\cal L}$-copositive, i.e., $x^\top Ax \ge0$ for all $x\in \partial{\cal L}$;
	\item[(iii)] There exists $\alpha\in {\mathbb R}$ such that   $A +\alpha{\rm J}$ is positive semidefinite;
	\item[(iv)] The function $f$ is bounded from below, % on ${{\mathbb H}^n}$, 
		i.e.,  there exist   an  $\alpha \in {\mathbb R}$ such that $f(p)\geq \alpha$, for all $p\in {{\mathbb H}^n} $.
\end{enumerate} 
\end{theorem}
\begin{proof}
We first prove that  (i) implies  (ii). Assume that $f$ is
hyperbolically convex. Let   $x\in \partial{\cal L}$  such that $x\neq 0$. Take
$y\in \partial{\cal L}$ not parallel to $x$, i. e., such that  $y\ne\alpha x$,
for all $\alpha\in {\mathbb R}$.  Define  the sequence $\lf(y^k\rg)_{k\in \mathbb N}$,
where $y^k:=(1/k)y$, for all $k\in \mathbb N$. Since  $y\in \partial{\cal L}$ and
$y\ne\alpha x$ for all $\alpha\in {\mathbb R}$,  we also have  $y^k\in
\partial{\cal L}$  and $y^k\ne\alpha x$, for all $\alpha\in {\mathbb R}$ and all
$k\in \mathbb N$. Thus, by using Lemma~\ref{le:ebc}, we conclude that   $x^\top
Jy^k<0$, for all $k\in \mathbb N$. Hence, considering that  $ x, y^k\in
\partial{\cal L}$ and  $x^\top Jy^k<0$, for all $k\in \mathbb N$, and $f$ is
hyperbolically convex, by  using Lemma~\ref{le:escc}, we  obtain that  $x^\top
Ax+ \lf(y^k\rg)^\top Ay^k\ge 0$, for all $k\in \mathbb N$.  Therefore,  owing to $y^k:=(1/k)y$,  for all $k\in \mathbb N$, we have
$$
x^\top Ax+\frac{1}{k^2} y^\top Ay\ge 0,  \qquad k\in \mathbb N.
$$
Therefore,  taking the limit  in the latter inequality, we obtain that $x^\top Ax \ge0$. In conclusion, $A$ is $\partial{\cal L}$-copositive.

Next we prove that  (ii) implies  (iii). Assume that  $A$ is $\partial{\cal
L}$-copositive. Thus,  for all $x\in {\mathbb R}^{n+1}$ with   $x^\top Jx=0$ we
have $x^\top Ax\ge0$. Consequently,   for all $x\in {\mathbb R}^{n+1}$ with
$x^\top Jx=0$ and $x\neq0$  we have  $x^\top (A+(1/k){\rm I})x>0$, for all $k\in
{\mathbb N}$. Hence, by Lemma~\ref{le:Finslert}, there exists $\alpha_k\in
{\mathbb R}$ such that $A+(1/k){\rm I} +\alpha_k{\rm J}$ is positive definite,
for all $k\in {\mathbb N}$.  We claim that the  sequence $(\alpha_k)_{k\in
{\mathbb N}}$ is  bounded. Indeed,  assume  by absurd  that   $(\alpha_k)_{k\in
{\mathbb N}}$ is  unbounded. Since  $A+(1/k){\rm I} +\alpha_k{\rm J}$ is positive
definite, for  each  $x\in {\mathbb R}^{n+1}$ with $x\neq0$ we conclude that 
$$
\frac{1}{\alpha_k}x^\top Ax+\frac{1}{k \alpha_k}x^\top x+x^\top{\rm J}x> 0, 
$$
for all $k\geq {\bar k}$ and some ${\bar k}\in {\mathbb N}$.  Hence, by taking the limit in the last inequality as $k$ goes to infinity, we conclude that $x^\top{\rm J}x\geq 0$ for all   $x\in {\mathbb R}^{n+1}$, which is absurd. Therefore, the claim is proved. Since the  sequence $(\alpha_k)_{k\in {\mathbb N}}$ is  bounded, we can  take a subsequence  $(\alpha_{k_j})_{j\in {\mathbb N}}$  and $\alpha\in {\mathbb R}$  such that  $\lim_{j\to \infty}\alpha_{k_j}=\alpha$. On the other hand, considering that  $A+(1/k){\rm I} +\alpha_k{\rm J}$ is positive definite for all $k\in {\mathbb N}$, we have 
$$
x^\top \left(A+\frac{1}{k_j}{\rm I}+{\alpha_{k_j}}{\rm J}\right)x> 0, 
$$
for all  $x\in {\mathbb R}^{n+1}$ such that $x\neq0$. Therefore, by taking the limit in the last
inequality as $k$ goes to infinity,  we have
$x^\top(A+ \alpha{\rm J})x\ge 0$, for all  $x\in {\mathbb R}^{n+1}$, which implies that  $A +\alpha{\rm J}$ is positive semidefinite.

%To end the proof,  we prove that  (iii) implies  (i).  
Next we prove that  (iii) implies  (i).  
Assume that there  is an
$\alpha\in {\mathbb R}$ such that   $A +\alpha{\rm J}$ is positive semidefinite. Due to  $A
+\alpha{\rm J}$  being  positive semidefinite and $ v^\top Av+ p^\top Ap=  v^\top Av+\alpha+ p^\top Ap-\alpha$, some calculations show
\begin{equation} \label{eq:fiic}
 v^\top Av+ p^\top Ap=   v^\top (A+\alpha{\rm J}) v+ p^\top (A+\alpha {\rm J})p\geq 0, 
\end{equation}
for all $p,v\in{\mathbb R}^{n+1}$ with $p^\top Jp=-1$,  $v^\top Jv=1$ and
$p^\top Jv=0$.  Therefore,  by using \eqref{eq:fiic} and Corollary~\ref{cor:lhc}, we conclude that $f$ is hyperbolically convex.

%First assume that  $f$ is hyperbolically convex. Thus, it follows from Theorem~\ref{th:pc} that there exists $\alpha\in {\mathbb R}$ such that    
Next we prove that  (iii) implies  (iv).  Assume that there exists $\alpha\in {\mathbb R}$ such that    
$A +\alpha{\rm J}$ is positive semidefinite. Hence,   $p^\top Ap +\alpha
p^\top{\rm J}p=p^\top(A +\alpha{\rm J})p\geq 0$, for all $p\in {{\mathbb R}^n}$.
Since $p\in {{\mathbb H}^n}$ implies $p^\top{\rm J}p=-1$,  we conclude that  $p^\top Ap \geq \alpha$, for all $p\in {{\mathbb H}^n} $. Therefore,  $f$ is bounded from below on ${{\mathbb H}^n}$.  

Finally, to conclude the proof, we prove that (iv) imply (iii). Assume that $f$ is bounded from below on ${{\mathbb H}^n}$. Thus, there exists an $\alpha \in {{\mathbb R}}$ such that 
$f(p)\geq \alpha$, for all $p\in {{\mathbb H}^n} $, or equivalently, 
\begin{equation} \label{eq:qfem}
p^T(A+\alpha J)p\geq 0, \qquad  \forall p\in {{\mathbb H}^n}.
\end{equation} 
 In order to apply Lemma~\ref{eq:slemma}, we will first prove the statement: if ${x}^{\top}{\rm J}{x}\leq 0$, then  ${x}^{\top}(A+\alpha J){x}\geq 0$. Let $x \in {{\mathbb R}^n}$ such that 
 ${x}^{\top}{\rm J}{x}\leq 0$. Take a sequence $(x_k)_{k\in {\mathbb N}} \in {{\mathbb R}^n}$ such that $\lim_{k\to +\infty}x_k=x$ and  ${x_k}^{\top}{\rm J}{x_k}< 0$. Define
\begin{equation} \label{eq:qfeme}
y^k=\frac{x_k}{\sqrt{-{x_k}^TJ{x_k}}}, \qquad k\in\N.
\end{equation} 
Since  ${x_k}^{\top}{\rm J}{x_k}< 0$, by using \eqref{eq:qfeme}, we conclude that
${y^k}^{\top}{\rm J}{y^k}=-1$. Thus,  $y^k\in {{\mathbb H}^n}$ and
\eqref{eq:qfem}  implies  that  $\lf(y^k\rg)^{\top}(A+\alpha J){y^k}\geq 0$. Using,
again \eqref{eq:qfeme} we obtain that  ${x_k}^{\top}(A+\alpha J){x_k}\geq 0$, for
all $k\in\N.$ Hence, by letting $k\to\infty$, we conclude that ${x}^{\top}(A+\alpha
J){x}\geq 0$, which proves the statement. Therefore, after applying
Lemma~\ref{eq:slemma}, we conclude that there exists a $\beta \in {\mathbb R}$ such that $A+(\alpha+\beta){\rm J}$ is positive definite. Hence, by Theorem~\ref{th:pc}, the function $f$ is hyperbolically convex.

\end{proof}
\begin{example}  
Let  $A \in {\mathbb R}^{(n+1)\times(n+1)}$ be a positive semidefinite matrix and $\alpha\in {\mathbb R}$.
Since   $ A= (A+\alpha{\rm J}) - \alpha{\rm J}$ is a positive semidefinite
matrix, by applying Theorem~\ref{th:pc}, we conclude that the function   $f_{\alpha}: {{\mathbb H}^n}: \to {\mathbb R} $ defined by $f_{\alpha}(p):=p^\top (A+\alpha{\rm J})p$  is  hyperbolically convex. 
\end{example}
For simplifying the statement and proof   of  the next results it is convenient to introduce the following notation. For a given  $A\in{\mathbb R}^{(n+1)\times(n+1)}$,  consider the following decomposition:
\begin{equation}\label{eq:not}
	A:=\left(
	\begin{matrix}
	{\bar A} & {a}\\
	{a}^\top & \sigma
	\end{matrix}
	\right),  \qquad   {\bar A}\in {\mathbb R}^{n\times n}, \quad {a}\in {\mathbb R}^{n\times 1}, \quad \sigma \in {\mathbb R} 
\end{equation} 
and denote by ${\bar I}\in{\mathbb R}^{n\times n}$ is  the identity matrix. For any $\alpha\in {\mathbb R}$, the decomposition \eqref{eq:not} yields 
\begin{equation} \label{eq:delt}
	        A+ \alpha J=\left(\begin{matrix}{\bar A}+\alpha \overline{I} &  a\\
	          {a}^\top  & \sigma -  \alpha \end{matrix}\right).
        \end{equation}
\begin{proposition} \label{pr:tccf}
Let  $A \in {\mathbb R}^{{(n+1)}\times {(n+1)}}$ be a symmetric matrix and  $\alpha\in {\mathbb R}$.   Consider  the decomposition \eqref{eq:not} and the  following statements: 
\begin{enumerate}
\item[(i)] The quadratic  function   $f: {{\mathbb H}^n} \to {\mathbb R} $ defined by $f(p):=p^\top Ap$  is  hyperbolically convex; 
\item[(ii)] The matrix $A +\alpha{\rm J}$ is positive definite; 
\item[(iii)]  The number $\alpha\in (-\lambda_{\min}({\bar A}), \sigma)$ and $\det(A+\alpha J)>0$;
\item[(iv)] The number $\alpha\in (-\lambda_{\min}({\bar A}), \sigma)$ and ${\sigma}-\alpha-{a}^\top \left({\bar A}+\alpha{\bar I}\right)^{-1}{a}>0$.
\end{enumerate}
Then    (ii), (iii)  and  (iv) are equivalent and any of them implies item~(i).
\end{proposition} 
\begin{proof}
We first prove the equivalence  between (ii) and (iii). Assume that (ii) holds.  Hence,  by applying    Lemmas~\ref{eq:posdef} and \ref{le:aspd}, and taking into account
  \eqref{eq:delt},  we have  $\det (A+ \alpha J)>0$, the matrix ${\bar
  A}+{\sigma}\overline{I}$  is  positive definite  and $ \sigma - \alpha>0$.
  Since  ${\bar A}+{\sigma}\overline{I}$  is   positive definite, we obtain that
  $\alpha> -\lambda_{\min}({\bar A})$. Therefore,  we conclude that  $\alpha \in
  ( -\lambda_{\min}({\bar A}),  \sigma)$ and $\det (A+ \alpha J)>0$. Hence, (iii)
  holds.  Reciprocally,    assume that   (iii) holds.  Thus, we have   $\alpha>
  -\lambda_{\min}({\bar A})$, $ \sigma > \alpha$ and  $\det(A +\alpha{\rm J})>0$.
  Hence,   the matrix ${\bar A}+{\sigma}{\bar I}$  is  positive definite,  $
  \sigma - \alpha>0$ and $\det(A +\alpha{\rm J})>0$. Therefore,  by using
  Lemmas~\ref{eq:posdef} and \ref{le:aspd} and the decomposition \eqref{eq:not},  we conclude that $A +\alpha{\rm J}$ is positive definite and  (ii) holds. 
  
Next,  we prove the equivalence  between (iii) and (iv).  First note that, for any  $\alpha\in (-\lambda_{\min}({\bar A}), \sigma)$,  we obtain that    $\lambda_{\min}({\bar A})+\alpha>0$. Thus,  the matrix ${\bar A}+\alpha {\bar I}$ is positive definite, which implies that $\det\left({\bar A}+\alpha{\bar I}\right)>0$. In particular,    the matrix ${\bar A}+\alpha {\bar I}$  is  invertible. Thus, applying  Lemma~\ref{le:sscdet}, we  have
		\begin{equation} \label{eq:itl}
			\det(A+\alpha J)=\left({\sigma}-\alpha-{a}^\top \left({\bar A}+\alpha{\bar I}\right)^{-1}{a}\right) \det\left({\bar A}+\alpha{\bar I}\right).
		\end{equation}
Since under the assumption $\alpha\in (-\lambda_{\min}({\bar A}), \sigma)$ we have  $\det\left({\bar A}+\alpha{\bar I}\right)>0$, it follows from \eqref{eq:itl}  that  $\det(A+\alpha J)>0$ is equivalent to ${\sigma}-\alpha-{a}^\top \left({\bar A}+\alpha{\bar I}\right)^{-1}{a}>0$.  Therefore, (iii) is equivalent to  (iv), and the proof of the  first statement of the proposition  is concluded.

By using the implication (iii)$\implies$(i) in Theorem~\ref{th:pc}, 
the last statement of the proposition follows from the first one.
\end{proof}
Consider   the decomposition \eqref{eq:not} of  a   symmetric matrix $A \in {\mathbb R}^{{(n+1)}\times {(n+1)}}$. Then,    it follows from Proposition~\ref{pr:tccf} that  we can decide if   $f(p):=p^\top Ap$  is  hyperbolically convex by solving the following optimization problem:
\begin{equation*}
            \inf \left\{ {\sigma}-\alpha-{a}^\top \left({\bar A}+\alpha{\bar I}\right)^{-1}{a}:~  \alpha\in (-\lambda_{\min}({\bar A}), \sigma) \right\}.             
\end{equation*}
 By using decompositions  \eqref{eq:dsv} and \eqref{eq:not},  Corollary \ref{cor:lhc}  can be stated equivalently in the following form.
\begin{corollary}\label{cor:det}
	Let $A=A^\top\in{\mathbb R}^{(n+1)\times(n+1)}$ and  $f:{\mathbb H}^n\to {\mathbb R}$ defined by
	$f(p)=p^\top Ap$. The function $f$ is hyperbolically convex if and 
	only if 
	\begin{equation}\label{eq:det}
		{\bar v}^\top{\bar A}{\bar v}+{\bar p}^\top{\bar A}{\bar p}
		+2v_{n+1}{\bar v}^\top {a}
		+2p^{n+1}{\bar p}^\top{a}+{\sigma}v_{n+1}^2
		+{\sigma}\left(p^{n+1}\right)^2\ge 0,
	\end{equation}
	for all $p,v\in{\mathbb R}^{n+1}$ with
	\begin{equation}\label{cond:det}
		{\bar p}^\top{\bar p}-\lf(p^{n+1}\rg)^2=-1
		\quad
		{\bar v}^\top{\bar v}-v_{n+1}^2=1,
		\quad 
		{\bar v}^\top{\bar p}-v_{n+1}p^{n+1}=0.
	\end{equation}
\end{corollary}
\begin{example}
	Let $A:=\lf(a_{i}^{j}\rg)\in{\mathbb R}^{(n+1)\times(n+1)}$ and  $f:{\mathbb H}^n\to {\mathbb R}$ be defined by $f(p)=p^\top Ap$. If $f$ is hyperbolically convex, then 
	\begin{equation} \label{eq:tem}
	\frac{1}{n}\sum_{i, j=1}^n a_{i}^{j} + {\sigma}\geq 0.
	\end{equation}
Indeed, take   ${\bar v}=(1/\sqrt{n}, \ldots,1/\sqrt{n})\in {\mathbb R}^{n}$,  $v_{n+1}=0$ and ${\bar p}=0\in {\mathbb R}^{n}$,  $p^{n+1}=1$. Thus,  \eqref{eq:det}  becomes 
$$
{\bar v}^\top{\bar A}{\bar v}+{\bar p}^\top{\bar A}{\bar p}
		+2v_{n+1}{\bar v}^\top{a}
		+2p^{n+1}{\bar p}^\top{a}+{\sigma}v_{n+1}^2
		+{\sigma}\lf(p^{n+1}\rg)^2= \frac{1}{n}\sum_{i, j=1}^n a_{i}^{j} + {\sigma}.
$$
Since $p=({\bar p},p^{n+1})\in{\mathbb R}^{n+1}$, $v=({\bar v},v_{n+1})\in{\mathbb R}^{n+1}$ satisfy  \eqref{cond:det} and considering that  $f$ is hyperbolically convex, the  inequality  \eqref{eq:tem} follows from applying Corollary~\ref{cor:det}.
\end{example}
\begin{theorem}  \label{th:pch}
Let $A\in{\mathbb R}^{(n+1)\times(n+1)}$ and $f:{\mathbb H}^n\to{\mathbb R}$  be defined by $f(p)=p^\top Ap$. Then, considering  decompositions  \eqref{eq:dsv} and \eqref{eq:not}, the following statements hold: 
\begin{enumerate}
	\item[(i)] If $f$ is hyperbolically convex, then $\lambda_{\min}({\bar A})\ge -{\sigma}$;
	\item[(ii)] If ${\sigma}\geq -\lambda_{\min}({\bar A})$ and ${a}=0$,  then $f$ is hyperbolically convex;
	\item[(iii)] If $\sigma+\lambda_{min}({\bar A})> 2\sqrt{a^\top a}$,  then $f$ is hyperbolically convex.
\end{enumerate}
\end{theorem} 
\begin{proof}
To prove (i), assume that $f$ is hyperbolically convex and  take any
$p,v\in{\mathbb R}^{n+1}$ with ${\bar p}=0$,  $p^{n+1}=1$, $v^{n+1}=0$, ${\bar
v}^\top{\bar v}=1$.  Then, conditions \eqref{cond:det} in Corollary~\ref{cor:det} are satisfied. Hence, it follows  from \eqref{eq:det} that ${\bar v}^\top{\bar A}{\bar v}\ge -{\sigma}$, for all ${\bar v}\in{\mathbb R}^n$ with  ${\bar v}^\top{\bar v}=1$. Therefore, the  result follows.

We proceed to prove (ii). Assume that $\lambda_{\min}({\bar A})\ge -{\sigma}$ and ${a}=0$. In this case,   we have   $\lambda_{\min}({\bar A}+{\sigma}\overline{I})\ge 0$, where ${\bar I}\in{\mathbb R}^{n\times n}$ is  the identity matrix. Hence, both matrices ${\bar A}+{\sigma}\overline{I}$ and
		\begin{equation*}
	        A+{\sigma}J=\left(\begin{matrix}{\bar A}+{\sigma}{\bar I} & 0\\
	      0 & 0\end{matrix}\right)
        \end{equation*}
		are positive semidefinite. Thus,  by applying Proposition
		\ref{pr:tccf} with $\alpha={\sigma}$, the proof of item (ii) follows.
		
To prove   item (iii),  	first we introduce the  auxiliary  quadratic polynomial $g:{\mathbb R}\to{\mathbb R}$ defined by
	\begin{equation*}
				g(t)=({\sigma}-t)\left({\lambda_{\min}({\bar A})}+t\right)-{a^\top}a.
        \end{equation*}
The roots of the quadratic polynomial $g$ are given by    
       \begin{equation*}
\mu={\frac{1}{2}\left(\sigma-{\lambda_{\min}({\bar A})}-\sqrt{\left[\sigma+{\lambda_{\min}({\bar A})}\right]^2-4{a^\top}a }\right)}, 
        \end{equation*}
         \begin{equation*}
\eta={\frac{1}{2}\left(\sigma-{\lambda_{\min}({\bar A})}+\sqrt{\left[\sigma+{\lambda_{\min}({\bar A})}\right]^2-4{a^\top}a }\right)}.
        \end{equation*}
Since $\sigma+\lambda_{min}({\bar A})> 2\sqrt{a^\top a}$  we have $\mu<\eta$.  Thus,  take  $\beta\in (\mu,\eta)$. Hence, $g(\beta)>0$. Since 
		\begin{equation*}
		\beta>\mu\ge{\frac{1}{2}\left(\sigma-{\lambda_{\min}({\bar A})}+\sqrt{\left[\sigma+{\lambda_{\min}({\bar A})}\right]^2}\right)} =-{\lambda_{\min}({\bar A})},
		\end{equation*} 
we obtain ${\lambda_{\min}({\bar A})}+\beta>0$. Thus, we conclude that  ${\bar A}+\beta {\bar I}$ is positive definite.  It follows that 
	  	\begin{align*}
		{\sigma}-\beta-{a}^\top \left({\bar A}+\beta{\bar I}\right)^{-1}{a}&\geq  {\sigma}-\beta-\lambda_{max}\left({\bar A}+\beta{\bar I}\right)^{-1}\,{a^\top}a\\
		                                                                                     &= {\sigma}-\beta-\frac{1}{{\lambda_{\min}({\bar A})}+\beta}\,{a^\top}a\\
		                                                                                     &=\frac{g(\beta)}{{\lambda_{\min}({\bar A})}+\beta}>0.
		\end{align*}
Therefore, by  using Lemma~\ref{le:sscdet}, it follows from the positive definiteness of the matrix ${\bar A}+\beta{\bar I}$ that 
		\begin{equation}\label{eq:l}
			\det(A+\beta J)=\left({\sigma}-\beta-{a}^\top \left({\bar A}+\beta{\bar I}\right)^{-1}{a}\right) \det\left({\bar A}+\beta{\bar I}\right)>0.
		\end{equation}
Since ${\bar A}+\beta{\bar I}$ is positive definite, combining
Lemma~\ref{eq:slemma} with  \eqref{eq:l}, we conclude that  $A+\beta J$ is positive definite.  Therefore,  Proposition \ref{pr:tccf} implies that  $f$ is hyperbolically convex, and the proof of  item~(iii) is concluded.

	\end{proof}
\begin{corollary}  \label{cr:a0}
Let $A\in{\mathbb R}^{(n+1)\times(n+1)}$ and $f:{\mathbb H}^n\to{\mathbb R}$  be defined by $f(p)=p^\top Ap$. Consider  the  decomposition  \eqref{eq:not} and assume that  ${a}=0$. Then,   $f$ is hyperbolically convex if and only if $\lambda_{\min}({\bar A})\ge -{\sigma}$.
\end{corollary} 
\begin{proof}
The proof is an immediate consequence of items (i) and (ii) of Theorem~\ref{th:pch}.

\end{proof}
In the next proposition we present a characterization for the case ${a}\neq 0$ in 
\eqref{eq:not}, which completes  the result of Corollary~\ref{cr:a0}.
\begin{proposition} 
Let  $A \in {\mathbb R}^{{(n+1)}\times {(n+1)}}$ be a symmetric matrix and the decomposition \eqref{eq:not}. If ${a}\neq 0$, then the following statements are
equivalent: 
\begin{enumerate}
\item[(i)] The quadratic  function   $f: {{\mathbb H}^n} \to {\mathbb R} $ defined by $f(p):=p^\top Ap$  is  hyperbolically convex; 
\item[(ii)] There exists $\alpha\in {\mathbb R}$ such that the matrix $A +\alpha{\rm J}$ is  positive semidefinite; 
\item[(iii)] There exists $\alpha\in {\mathbb R}$ such that $\sigma>\alpha$ and the matrix 
	${\bar A}+\alpha {\bar I} -\frac{1}{{\sigma}-\alpha}aa^\top$ is positive 
	semidefinite;
\end{enumerate}
\end{proposition} 
\begin{proof}
	First we prove that (ii) and (iii) are equivalent.
 Since the matrix $A +\alpha{\rm J}$ is positive semidefinite, by using the decompositions  \eqref{eq:delt} and \cite[Corollary 7.15, p. 398]{HornJohnson1990},  we conclude   that 
\begin{equation} \label{eq:pdif}
({\bar a}_{ii}+\alpha)(\sigma-\alpha)\geq a_i^2, \qquad i=1, \ldots, n, 
\end{equation}
where ${\bar a}_{ii}$ is the $ii$-entry of the matrix ${\bar A}$ and $a_i$ is the  $i$-entry of the vector $a$. Considering that ${a}\neq 0$, and all elements in the diagonal of a positive semidefinite matrix  are  nonnegative, it follows from \eqref{eq:pdif} that $\sigma>\alpha$. Thus, applying item (iii) of Lemma~\ref{le:aspd} we conclude  that items  (ii) and (iii) are equivalent.  
The equivalence of (i) and (ii) follows from the equivalence of  (iii) and (i) in Theorem~\ref{th:pc}.

\end{proof}
\begin{proposition} 
Let  $A \in {\mathbb R}^{{(n+1)}\times {(n+1)}}$ be a symmetric matrix and the decomposition \eqref{eq:not}.  Consider the following statements: 
\begin{enumerate}
\item[(i)] The quadratic  function   $f: {{\mathbb H}^n} \to {\mathbb R} $ defined by $f(p):=p^\top Ap$  is  hyperbolically convex; 
\item[(ii)] There exists $\alpha\in {\mathbb R}$ such that ${\bar
	A}+\alpha {\bar I}$ is positive definite and $A +\alpha{\rm J}$ is  positive 
	semidefinite; 
\item[(iii)] There exists $\alpha\in {\mathbb R}$ such that ${\bar
	A}+\alpha {\bar I}$ is positive definite and
	${\sigma}-\alpha-a^\top({\bar A}+\alpha {\bar I} )^{-1}a\geq 0$.
\end{enumerate}
Then, items (ii) and (iii) are equivalent and any of them implies (i). 
\end{proposition} 
\begin{proof}
It follows from item  (iii)  of Lemma~\ref{le:aspd} that   items   (ii) and  (iii) are equivalent. The equivalence of (i) and (ii) follows from the equivalence of  (iii) and (i) in Theorem~\ref{th:pc}.

\end{proof}
\begin{proposition}
Let  $A \in {\mathbb R}^{{(n+1)}\times {(n+1)}}$ be a symmetric matrix and the decomposition \eqref{eq:not}.  Consider the   following statements: 
\begin{enumerate}
\item[(i)] The quadratic  function   $f: {{\mathbb H}^n} \to {\mathbb R} $ defined by $f(p):=p^\top Ap$  is  hyperbolically convex; 
\item[(ii)] There exists $\alpha\in {\mathbb R}$ such that the matrix $A +\alpha{\rm J}$ is positive definite; 
\item[(iii)] There exists $\alpha\in {\mathbb R}$ such that $\sigma>\alpha$ and the matrix ${\bar A}+\alpha {\bar I} -\frac{1}{{\sigma}-\alpha}aa^\top$ is positive definite.
\end{enumerate}
Then, items    (ii) and (iii)  are equivalent and any of them implies item~(i).
\end{proposition} 
\begin{proof}
The equivalence between items   (ii)   and  (iii) follows by direct application of item  (i) of Lemma~\ref{le:aspd}.  
By using the implication (iii)$\implies$(i) in Theorem~\ref{th:pc}, 
the last statement of the proposition follows from the first one.

\end{proof}
%%%%%%%%%%%%%%%%%%%%%%%%%%%%%%%%%%%%%%%%%%%%%%%%%%%
\section{Optimization Concepts on the Hyperbolic Space} \label{sec:fr2}
In this section we present some concepts of optimization related to
hyperbolically convex function. In order to do that, consider a  differentiable
function $f:{\mathbb H}^n \to {\mathbb R}$  and the following  unconstrained 
optimization problem
\begin{equation} \label{eq:uncop}
\quad \displaystyle {\rm Minimize}_{p\in {\mathbb H}^n} f(p).
\end{equation}
It follows from \eqref{eq:grad} that a necessary optimality condition for the unconstrained  
problem~\eqref{eq:uncop} is:
\begin{equation} \label{eq:opup}
\left[{\rm I}+pp^{\top}{\rm J} \right] {\rm J}\cdot Df(p)=0.
\end{equation}
\begin{remark}
If $f$ is  a hyperbolically  convex function, then Proposition~\ref{pr:cfocf} implies that 
all points satisfying \eqref{eq:opup} are global solutions of problem~\eqref{eq:uncop}, i.e., \eqref{eq:opup}  is also  a sufficient  optimality condition. 
\end{remark}
Let ${\cal C} \subset {\mathbb H}^n$ be a hyperbolically convex set and consider the constrained optimization  problem:
\begin{equation} \label{eq:op}
\quad \displaystyle {\rm Minimize}_{p\in {\cal C}} f(p).
\end{equation}
In the following proposition we state the necessary optimality condition for the  problem~\eqref{eq:op}.
\begin{proposition} \label{pro:ncvi}
It  the point  \(\bar p\in {\cal C}\)  is a solution of problem \eqref{eq:op}, then
$$
 \Big\langle  [{\rm I}+{\bar p}{\bar p}^{\top}{\rm J}] {\rm J}\cdot Df (\bar p),p \Big\rangle \geq 0, \quad \forall ~p\in {\cal C}, 
$$
where $Df(\bar p) \in {\mathbb R}^{n+1}$ is the usual gradient of  $ f$ at $\bar p$. 
\end{proposition}
\begin{proof}
Take  \( p\in {\cal C}\) and let \(\bar p\in {\cal C}\) be a solution to  \eqref{eq:op}.  Let $ [0, 1]\ni t \mapsto  \gamma_{\bar p p}(t)= \exp_{\bar p} (t \,\log_{\bar p}p),$
be the geodesic from \(\bar p\) to \(p\). Since \({\cal C}\) is hyperbolically convex and \(  p,\bar p\in {\cal C}\), we 
conclude that \( \gamma_{\bar p p}(t)\in {\cal C}\) for all \(t\in [0, 1]\). 
Hence,  as \(\bar p\in {\cal C}\) is a solution to the problem in \eqref{eq:op},
we have $(f(\gamma_{\bar p p}(t))-f(\bar p))/t\geq 0$,  for all $t\in [0, 1].$
Thus, taking the limit when $t$ tends to zero, we obtain, by using
\eqref{eq:cr1} and  $\gamma_{\bar p p}'(0)= \log_{\bar p}p$, that
$\langle \grad f (\bar p), \log_{\bar p}p\rangle \geq 0.$  Therefore,
the result follows from  \eqref{eq:expinv} and \eqref{eq:grad}, by taking into account that  
${\rm arcosh}(-\langle \bar p, p\rangle)\geq 0 $.  

\end{proof}
\begin{proposition} \label{pro:cvi}
Let  \( f\) be a hyperbolically  convex function in \({\cal C}\).  The point \(\bar p\in {\cal C}\)  is a solution of  the problem in \eqref{eq:op} if and only if
$$
 \Big\langle [{\rm I}+{\bar p}{\bar p}^{\top}{\rm J}] {\rm J}\cdot Df (\bar p), p \Big\rangle \geq 0, \quad \forall ~p\in {\cal C},
$$
where $Df(\bar p) \in {\mathbb R}^{n+1}$ is the usual gradient of  $ f$ at $\bar p$. 
\end{proposition}
\begin{proof}
If  the point \(\bar p\in {\cal C}\)  is a solution of \eqref{eq:op}, then the inequality 
follows from Proposition~\ref{pro:ncvi}. Conversely,  take  $p, \,\bar p\in {\cal C}\), \( p\neq \bar p$ 
and assume that $ \big\langle [{\rm I}+{\bar p}{\bar p}^{\top}{\rm J}]{\rm J}\cdot Df (\bar p),  p \big\rangle\geq 0 $. 
As \(f\) is hyperbolically convex  in \({\cal C}\), we conclude from Proposition~\ref{pr:cfocf} that
\[
f(p)\geq f(\bar p)+ \frac{{\rm arcosh}(-\langle \bar p, p\rangle)}{\sqrt{1-\langle \bar p, p\rangle ^2}} 
 \Big\langle [{\rm I}+{\bar p}{\bar p}^{\top}{\rm J}]{\rm J}\cdot Df (\bar p),  p \Big\rangle, \qquad \forall \, p\in {\cal C}, \; p\neq \bar p.
\]
Since $ \big\langle  [{\rm I}+{\bar p}{\bar p}^{\top}{\rm J}] {\rm J}\cdot Df (\bar p),
p \big\rangle\geq 0 $ and ${\rm arcosh}(-\langle \bar p, p\rangle)\geq 0$, the  latter inequality  implies that \(f(p)\geq f(\bar p)\), for all  \(p\in {\cal C}\). Therefore, \(\bar p\)  is a global solution of the problem in \eqref{eq:op}.  

\end{proof}
Next  we present an equivalent  form for  Proposition~\ref{pr:conv}, whose proof  follows by combining  Proposition~\ref{pr:sgmd2}  with  Lemma~\ref{le:csqd},  Proposition~\ref{pro:cvi} and \eqref{eq:ddf},
\begin{corollary} \label{cr:pcha}
Let ${\cal C\subseteq} {{\mathbb H}^n}$  be a closed hyperbolically convex set and  ${\bar p}\in {{\mathbb H}^n}$.  Consider  the function ${{\mathbb H}^n}\ni p\mapsto \rho_{\bar p}(p):= \frac{1}{2}d^2_{\bar p}(p)$  defined in \eqref{eq:sqdist}. Then, 
${\cal P_{\cal C}}({\bar p})= \arg\min_{p\in {\cal C}} \rho_{\bar p}(p) $ if and only if 
 $$
 \left\langle \left({\rm I}+{\cal P_{\cal C}}({\bar p}) {\cal P_{\cal C}}({\bar p})^\top{\rm J}\right){\bar p},  p\right\rangle \leq 0, \quad  \qquad  ~ \forall  p\in {\cal C}.
 $$
\end{corollary}
For $f:{\mathbb H}^n \to {\mathbb R}$ and $g_i:{\mathbb H}^n \to {\mathbb R}^m$, $i=1, \ldots, m$ differentiable hyperbolically convex  
functions and 
$$
{\cal C} =\{p\in {\mathbb R}^n: g_i(p)\leq 0, i=1, \ldots, m\} 
$$
a hyperbolically convex set, see Remark~\ref{re:pohcf}. Consider the following particular instance of the hyperbolically convex
optimization problem~\eqref{eq:op}:
\begin{equation} \label{eq:opc}
\quad \displaystyle {\rm Minimize}_{x\in {\cal C}} f(x).
\end{equation}
\begin{proposition} \label{pro:opt} 
Suppose that ${\bar p} \in {\cal C} $ and  there exists $\mu=(\mu_{1}, \ldots,\mu_{m})\in {\mathbb R}^m_{+}$ such that
\begin{equation} \label{kkts}
\left[{\rm I}+{\bar p}{{\bar p}}^{\top}{\rm J}\right]\Big[{\rm J}\cdot Df({\bar p})+\sum_{i=1}^{m}\mu_{i}{\rm J}\cdot Dg_i({\bar p})\Big]=0,   \qquad  \sum_{i=1}^{m}\mu_{i}g_{i}({\bar p})=0, 
\end{equation}
where $Df({\bar p}) \in {\mathbb R}^{n+1}$ is the usual gradient of  $ f$ at ${\bar p}$. Then ${\bar p}$ is a solution of the problem~\eqref{eq:opc}.
\end{proposition}
\begin{proof}
Since  $f, g_{i}:{\mathbb H}^n \to {\mathbb R}$ are  hyperbolically convex
functions and $\mu_i\ge 0$, for $i=1, \ldots, m$, it follows that $h:{\mathbb H}^n \to {\mathbb R}$ defined by 
$$
h(p)=f(p)+\sum_{i=1}^{m}\mu_{i}g_{i}(p)
$$
 is hyperbolically convex. Moreover, $f(p)\geq h(p)$, for all $p\in {\cal C}$. Now,  from the  first equality in \eqref{kkts} we obtain that 
 $\left[{\rm I}+pp^{\top}{\rm J} \right] {\rm J}\cdot Dh({\bar p})=0$. Thus, since $h$
 is hyperbolically convex, we can apply Proposition~\ref{pro:cvi} with $f=h$ and
 ${\cal C}={\mathbb R}^{n+1}$ to conclude that  ${\bar p}$ is a minimizer of $h$
 in ${\mathbb R}^{n+1}$.  Hence, from $f(p)\geq h(p)$, for all $p\in {\cal C}$
 and the second equality in \eqref{kkts}, we have  
$
f(p)\geq h(p)\geq h({\bar p})=f({\bar p}),
$
for all $p\in {\cal C} $. Since ${\bar p} \in {\cal C} $, the last inequality implies that  it is  a solution  of \eqref{eq:opc}.

\end{proof}
\begin{remark}
If ${\bar p}$ is a solution of the problem~\eqref{eq:opc}, then under mild conditions on the problem~\eqref{eq:opc} the point   ${\bar p}$ satisfies \eqref{kkts}, for details  see  \cite[Section~9]{Udriste1994}; see also \cite[Theorem 4.4]{YangZhangRuyi2014}.
\end{remark}
%%%%%%%%%%%%%%%%%%%%%%%%%%%%%%%%%%%%%%%%%%%%%%%%%%%
\section{Convexity on Others Model of Hyperbolic Geometry} \label{sec:gm}
There  are several models of hyperbolic geometry, the  four commonly used   ones    are  {\it the Klein model},    {\it the  Poincar\'e disk model}, {\it the Poincar\'e half-plane model}  and   {\it the Lorentz or hyperboloid model}, see for example  \cite{BenedettiPetronio1992, Cannon1997, Ratcliffe2019}.  Among them, we have chosen the hyperboloid model, because  it has several similarities  with the Euclidean sphere. It is worth noting that we can choose any of the aforementioned models. Let  us recall the general concept of an  isometry. 
\begin{definition} \label{def:iso}
Let  $(\mathcal{N},\langle\!\!\langle \cdot ~, ~ \cdot \rangle\!\!\rangle)$ and $(\mathcal{M},  \langle \cdot ~, ~ \cdot \rangle)$ be   Riemannian manifolds. A mapping  $\Phi : \mathcal{N} \to \mathcal{M}$ is called  an  isometry,  if $\Phi$  is continuously differentiable, and for all $q \in \mathcal{N}$ and $u, v \in T_q\mathcal{N}$, we have 
$
\langle\!\!\langle u , v \rangle\!\!\rangle =\langle d\Phi_q u , d\Phi_q v \!\rangle, 
$
 where $d\Phi_q : T_{q}\mathcal{N} \to T_{\Phi(q)}\mathcal{M}$ is the differential of $\Phi$ at $q\in \mathcal{N}$.  
 \end{definition}  The next result is an important property of isometries, its prove is in \cite[Proposition 5.6.1, p. 196]{Petersen2016}.
\begin{proposition}\label{prop:geod.iso} 
Let $\mathcal{N}$, $\mathcal{M}$ be Riemannian manifold and  $\Phi : \mathcal{N} \to \mathcal{M}$  an isometry. If $\gamma$ is a geodesic in ${\mathcal N}$, then $\Phi\circ\gamma$ is a geodesic in ${\mathcal M}$. Moreover, $\Phi$ preserve the Riemannian distance.  
\end{proposition}
Straight combination of Definition~\ref{def:iso}  with  Proposition~\ref{prop:geod.iso} give us the following result.
\begin{theorem}\label{prop:iso.conv.fun}
Let $\mathcal{N}$, $\mathcal{M}$ be Riemannian manifolds and the function   $\Phi : \mathcal{N} \to \mathcal{M}$  be an isometry. The function $g:{\mathcal M}\to {\mathbb R}$  is convex if and only if $f:\mathcal{N} \to {\mathbb R}$ defined by
$f(p)=(g\circ\Phi)(p)$ is convex.
\end{theorem}
It is well known that  the Klein model,   the  Poincar\'e disk model,  the Poincar\'e half-plane model and   the hyperboloid model are isometric to each other,  see       isometries between  them  in \cite[Chapter A]{BenedettiPetronio1992}.   Therefore, it follows from Theorem~\ref{prop:iso.conv.fun} that the concepts of convexity and consequently the results studied in the previous sections have via isometries  their counterparts  in any  model isometric to the  hyperboloid model.
%%%%%%%%%%%%%%%%%%%%%%%%%%%%%%%%%%%%%%%%
\section{Final Remarks} \label{sec:fr}
This paper is inspired by the papers \cite{FerreiraIusemNemeth2014,
FerreiraNemeth2019}, where we studied some intrinsic properties of the
spherically  convex functions and  spherically  quadratic functions,
respectively. Despite some of our ideas being similar with the ideas in the
aforementioned papers, the convex functions in hyperbolic spaces turned out to have a 
completely different structure. For example there is no constant globally convex
function on the whole sphere, but there are many such functions on the hyperbolic
space, as shown in Section~{5.2}. A related remark is that the class of convex functions on  
constant curvature manifolds is widening with the decrease of the sign of the curvature. 
We also expect that the class of convex functions on a proper convex subset of the 
hyperbolic space to be much more wider than the convex functions on the corresponding proper 
convex subset of the sphere. This property has already been established for
corresponding intersections of the sphere and hyperbolic space with the positive
orthant \cite{FerreiraNemeth2019}, but for other cones it needs to be 
investigated.  Although several applications of optimization in hyperbolic
spaces have emerged, a comprehensive study from this point of view of these
spaces is still 
lacking. The results of this paper are the first step in this direction. We 
foresee significant progress in this topic in the nearby future.

Finally,  let us present some  basics formulas similar to the ones in
Section~\ref{sec:int.1}  for  an   $n$-dimensional  hyperbolic space  with
constant negative curvature $-K<0$. For that,   let us rescale the Lorentzian inner product  \eqref{eq:ip}  as follows: Let $K>0$ and  $\langle\cdot , \cdot \rangle_{K}$ be the {\it $K$-Lorentzian inner product}  of  $x:=(x_1, \ldots,x_n, x_{n+1})^{\top} $ and  $y:=(y_1, \ldots, y_n,y_{n+1})^{\top}$ on ${{\mathbb R}^{n+1}}$ defined  by 
\begin{equation} \label{eq:ipg}
{\langle} x, y{\rangle_{K}}:= Kx_{1}y_{1}+ \cdots + Kx_{n}y_{n}-Kx_{n+1}y_{n+1}.
\end{equation}
Note that ${\langle} x, y{\rangle_{K}}={K}{\langle} x, y{\rangle}$. For each  $x\in {{\mathbb R}^{n+1}}$,  the {\it ${K}$-Lorentzian norm (length)} of $x$ is  the complex number
\begin{equation*} 
\|x\|_{K}:= \sqrt{{\langle} x, x{\rangle_{K}}}.
\end{equation*}
In order to state the inner product \eqref{eq:ipg} in a convenient form, we  take the diagonal matrix ${\rm J}_{K}$ 
defined  by
\begin{equation} \label{eq:dmipg}
{\rm J}_{K}:={\rm diag}({K}, \ldots,{K}, -{K}) \in {\mathbb R}^{(n+1)\times (n+1)}.
\end{equation}
By using \eqref{eq:dmipg},   the  Lorentz inner product  \eqref{eq:ipg}  can be
written equivalently  as follows 
\begin{equation*}
{\langle} x, y{\rangle}_{K}:=x^{\top}{\rm J}_{K} y, \qquad \forall x, y\in  {{\mathbb R}^{n+1}}.
\end{equation*}
By considering  the ${K}$-Lorentzian inner product. $\langle\cdot , \cdot
\rangle_{K}$ defined in \eqref{eq:ipg},  we define  the {\it   $n$-dimensional  ${K}$-hyperbolic space}  as follows 
\begin{equation*}
{\mathbb H}^{n}_{K}:=\left\{ p\in {{\mathbb R}^{n+1}}:~\langle p, p\rangle_{K}
=-1, ~ p^{n+1}>0\right\}, \qquad {K}>0.
\end{equation*}
It is worth noting that the  $n$-dimensional  ${K}$-hyperbolic space ${\mathbb
H}^{n}_{K}$ can also be written as follows 
\begin{equation} \label{eq:hsg2}
{\mathbb H}^{n}_{K}:=\left\{ p\in {{\mathbb R}^{n+1}}:~\langle p, p\rangle
=-\frac{1}{K}, ~ p^{n+1}>0\right\}, \qquad {K}>0.
\end{equation}
We know that ${\mathbb H}^{n}_{K}$ has sectional curvature $-K$.  It follows from  \eqref{eq:hsg2} that ${\mathbb H}^{n}_{1}$ is the  $n$-dimensional hyperbolic space ${\mathbb H}^{n}$. The tangent plane  of ${\mathbb H}^{n}_{K}$ at a point  $p\in {\mathbb H}^{n}_{K}$ is given by 
\begin{equation*} 
T_{p}{{\mathbb H}^n_{K}}:=\left\{v\in {{\mathbb R}^{n+1}}\, :\, \langle p, v \rangle_{K}=0\right\}.
\end{equation*}
The  {\it intrinsic distance} on the  ${K}$-hyperbolic space ${\mathbb H}^{n}_{K}$ between two  points $p, q \in {{\mathbb H}^n_{K}}$  is  given  by
\begin{equation*}
d^{K}(p, q):=\frac{1}{\sqrt{K}}{\rm arcosh} (-\langle p , q\rangle_{K}).
\end{equation*}
If $ p, q\in {{\mathbb H}^n}$ and  $q\neq p$, then the   unique {\it geodesic segment  from $p$ to $q$ } is   given by 
\begin{equation*}
\gamma^{K}_{pq}(t)= \left( \cosh({\sqrt{K}} t) + \frac{\langle p, q\rangle_{K} \sinh
({\sqrt{K}}t)}{\sqrt{\langle p, q\rangle_{K}^2-1}}\right) p 
+ \frac{\sinh({\sqrt{K}} t)}{\sqrt{\langle p, q\rangle_{K}^2-1}}\;q, \qquad \forall t\in
\lf[0, \;d^{K}(p,q)\rg].
\end{equation*}
The {\it exponential mapping} $\exp^{K}_{p}:T_{p}{{\mathbb H}^n_{K}} \rightarrow {{\mathbb H}^n_{K}} $ at a points $p \in {{\mathbb H}^n_{K}}$ is given by 
\begin{equation*} 
\exp^{K}_{p}v:= \displaystyle \cosh(\|v\|_{K}) \,p+ \sinh(\|v\|_{K})\, \frac{v}{\|v\|_{K}},
\qquad  \forall v\in T_p{{\mathbb H}^n_{K}}\setminus\{0\}.
\end{equation*} 
If  $\gamma^{K} _{v}$ is the geodesic defined by its  initial position $p$, with
velocity $v$ at $p$, then  $\gamma^{K} _{v}(t)= \exp^{K}_{p}tv$. The {\it inverse
of the exponential mapping}  is given by $\log_{p}^{K}q=0$, for $q=p$, and 
\begin{equation} \label{eq:expinvg}
\log_{p}^{K}q:=  \displaystyle \frac{1}{\sqrt{K}} {\rm arcosh}(-\langle p, q\rangle_{K}) \frac{1}{\sqrt{\langle p, q\rangle_{K}^2-1}} 
\left[{\rm I}+pp^\top{\rm J}_{K}\right]q,  \qquad  q\neq p.
\end{equation}
It follows from \eqref{eq:Intdist} and \eqref{eq:expinv} that $d_{K}(p,
q)=\|\log_{q}^{K}p\|_{K}$, for all $p, q \in {{\mathbb H}^n_{K}}$. The explicitly formula of parallel transport $P^{K}_{pq}$ is given by 
\begin{equation*}
{P}^{K}_{pq}(v):= v-\frac{\langle v, \log_{q}^{K}p \rangle_{K}}{{\rm arcosh}^2 (-\langle p , q\rangle_{K})}\left(\log_{q}^{K}p+ \log_{p}^{K}q \right)=\left[{\rm I}+\frac{1}{1-\langle p, q\rangle_{K}}(p+q)q^\top {\rm J}_{K}\right]v.
\end{equation*}
 By rescaling the   Lorentzian inner product  \eqref{eq:ip}
 to \eqref{eq:ipg},  we can obtain similar results to the previous sections. It
 follows from a rescaled version of Theorem \ref{th:pc} that if a quadratic function is 
 $K_0$-hyperbolic convex for a $K_0>0$, then it is $K$-hyperbolic convex with respect to all 
 $K>0$ (where $K$-hyperbolic convexity in ${{\mathbb H}^n_{K}}$ can be defined
 similarly to hyperbolic convexity in ${{\mathbb H}^n})$. 
 However, the only $K$-hyperbolocally convex quadratic functions that remain convex when 
 $K\to 0$ are the ones with $A$ positive semidefinite. It is interesting to
 study a similar question for more general functions. 

\subsection*{Acknowledgements}
The authors are grateful to Andr\'as R\'acz for suggesting the topic of convexity of sets and functions on hyperbolic spaces. They would also like to thank the reviewers for their remarks and suggestions which contributed to the improvement of the paper.

\end{document}